\documentclass{amsart}

\usepackage{graphicx}
\usepackage{amsmath}
\usepackage{amssymb}
\usepackage{textcomp}

\newtheorem{thm}{Theorem}[section]
\newtheorem{prop}[thm]{Proposition}
\newtheorem{lemma}[thm]{Lemma}

\newtheorem{conj}{Conjecture}
\newtheorem*{thm*}{Theorem}
\newtheorem{rem}{Remark}

\newcommand{\h}{\tfrac{1}{2}}
\newcommand{\R}{{\mathbb R}}

\begin{document}

\title{On the Correlation of Shifted values of the Riemann Zeta Function}

\author{Vorrapan Chandee}
\email{vchandee@stanford.edu}

\address{Department of Mathematics, Stanford University, Stanford, CA 94305}

\date{\today}

\begin{abstract}
In 2007, assuming the Riemann Hypothesis (RH), Soundararajan \cite{Moment} proved that $
\int_{0}^T |\zeta(1/2 + it)|^{2k} \> dt \ll_{k, \epsilon} T(\log T)^{k^2 + \epsilon}$
for every $k$ positive real number and every $\epsilon > 0.$ In this paper I generalized his methods to find upper bounds for shifted moments. We also obtained their lower bounds and conjectured asymptotic formulas based on Random matrix model, which is analogous to Keating and Snaith's work. These upper and lower bounds suggest that the correlation of $|\zeta(\h + it + i\alpha_1)|$ and $|\zeta(\h + it + i\alpha_2)|$ transition at $|\alpha_1 - \alpha_2| \approx \frac{1}{\log T}$. In particular these distribution appear independent when $|\alpha_1 - \alpha_2|$ is much larger than $\frac{1}{\log T}.$
\end{abstract}

\maketitle
\section{Introduction}
Finding moments of the Riemman zeta function $\zeta(s)$ is an important problem in analytic number theory, especially the moments on the critical line:
$$ M_k(T) := \int_0^T \big| \zeta(\h + it) \big|^{2k} \> dt. $$
Extensive work has been done to find an asymptotic formula for $M_k(T)$; however, the only unconditional results in this direction are proven for $k =1,$ due to Hardy and Littlewood, and k = 2, due to Ingham \cite{Ti}. Assuming the Riemann hypothesis (RH), good upper and lower bounds are available. Ramachandra \cite{Rama} proved that for any positive real integer $k$, $ M_k(T) \gg T\log^{k^2} T.$ Later in 2007, Soundararajan \cite{Moment} showed that for every positive real number $k$ and every $\epsilon > 0$ 
\begin{equation} \label{eqn:uppSound}
M_k(T) \ll_{k,\epsilon} T(\log T)^{k^2 + \epsilon}.
\end{equation} 
In 2000, Keating and Snaith \cite{KS} conjectured an asymptotic formula for $M_k(T)$, for every positive integer $k$, based on the random matrix model for the zeros of $\zeta(s).$ They suggested that the value distribution of $\zeta(1/2 + it)$ is related to that of the characteristic polynomials of random unitary matrices, $\Lambda(e^{i\theta}) := \prod_{n=1}^N (1 - e^{i(\theta_n - \theta)}).$ Therefore they computed the moments of the characteristic polynomials to arrive at a conjecture for $M_k(T)$ and showed that 
\begin{equation} \label{eqn:moment_rmt}
g_U(N,k) := \int_{U(N)} |\Lambda(e^{i\theta})|^{2k} \> dU_N \sim \frac{G^2(k+1)}{G(2k+1)} N^{k^2},
\end{equation}
where $G$ is the Barnes G-function.  
Using the scaling $N = \log \tfrac{T}{2\pi},$ this led them to conjecture that $M_k(T) \sim a(k)\frac{G^2(k+1)}{G(2k+1)}T\log^{k^2}T,$ where 
\begin{equation} \label{eqn:ak}
a(k) := \prod_p \big(\big(1-\tfrac{1}{p} \big)^{k^2} \sum_{m=0}^{\infty} \big(\tfrac{\Gamma(m+k)}{m!\Gamma(k)}\big)^2 p^{-m}\big).
\end{equation}
This conjecture agrees with the known results for $k=1, 2.$ 
\\
\\
A generalization of the moments of $\zeta(s)$ are the shifted moments, defined as
\begin{equation} \label{eqn:Mkt}
M_{\bf k}(T,{\overrightarrow{\alpha}}) = \int_0^T |\zeta(\h + it + i\alpha_1)|^{2k_1}|\zeta(\h + it + i\alpha_2)|^{2k_2} ... |\zeta(\h + it + i\alpha_m)|^{2k_m} \> dt,
\end{equation}
where ${\bf k} = (k_1, k_2,..., k_m)$ is a sequence of positive real numbers and ${\overrightarrow{\alpha}} = (\alpha_1,...\alpha_{m}),$ where $\alpha_i \neq \alpha_j$ when $i \neq j,$ $|\alpha_i - \alpha_j| = O(1), $ and $\alpha_i = O(\log T).$ Also $\alpha_i = \alpha_i(T)$ is a real valued function in terms of T such that $\lim_{T \rightarrow \infty} \alpha_i \log T$ and  $\lim_{T \rightarrow \infty} (\alpha_i - \alpha_j) \log T$ exists or equals $\pm \infty.$

Conrey, Farmer, Keating, Rubinstein and Snaith \cite{CFKRS} gave a general recipe from which an asymptotic formula for the shifted moments of the Riemann zeta function may be conjectured. However, it is not immediately clear from their recipe what the leading asymptotic term for these shifted moments should be, and this is elucidated by K\"{o}sters in \cite{Kshift}.  Specifically, based on the work in \cite{CFKRS}, K\"{o}sters conjectures that for any $T_0 > 1$ and $\mu_1,..., \mu_M \in {\mathbb R},$
\begin{eqnarray} \label{eqn:conjshift}
&& \lim_{T \rightarrow \infty} \frac{1}{T(\log T)^{M^2}} \int_{T_0}^T |\zeta(\h + it + i\tfrac{2\pi\mu_1}{\log T})|^{2} ... |\zeta(\h + it + i\tfrac{2\pi\mu_M}{\log T})|^{2} \> dt \\
&=& \frac{a(M)}{\Delta^2(2\pi\mu_1,...,2\pi\mu_M)} \det (b_{jk})_{j,k = 1,...,M}, \nonumber  
\end{eqnarray}       
where $\Delta(x_1,..,x_n) = \prod_{1 \leq j < k \leq n} (x_k - x_j),$ $a(M)$ is defined in (\ref{eqn:ak}),  and $ b_{jk} = \frac{\sin \pi (\mu_j - \mu_k)}{ \pi (\mu_j - \mu_k)}$ if $j \neq k,$ and 1 otherwise. Note that in the case where two or more of $\mu_i's$ are equal, the right hand side is defined as the continuous extension.

Inspired by the above work, we are interested in studying the shifted moments $M_{(k,k)}(T, (\alpha_1, \alpha_2)),$ where $k$ is a positive integer. This will help us to understand the correlation between the values of $\zeta(\h + it + i\alpha_1)$ and $\zeta(\h + it + i\alpha_2).$  As stated at the beginning, it is difficult to compute $M_{(k,k)}(T, (\alpha_1, \alpha_2)).$ Hence we will start by formulating a conjecture for its asymptotic formula based on Keating and Snaith's random matrix model. Specifically, the leading asymptotic term of $M_{(k,k)}(T, (\alpha_1, \alpha_2))$ is as follows
\begin{conj} \label{eqn:conjcor}
$$
M_{(k,k)}(T, (\alpha_1, \alpha_2)) \left\{ \begin{array}{lll}
                    \sim_{k} & T \log^{4k^2} T & {\rm if} \,\,\, \lim_{T \rightarrow \infty} |\alpha_1 - \alpha_2|\log T = 0, \\
\sim_{k, c} & T \log^{4k^2} T & {\rm if} \,\,\, \lim_{T \rightarrow \infty} |\alpha_1 - \alpha_2|\log T = c \neq 0. \\
                     \sim_{k} & \frac{1}{|\alpha_1 - \alpha_2|^{2k^2}}T \log^{2k^2}T & {\rm if} \,\,\, \lim_{T \rightarrow \infty} |\alpha_1 - \alpha_2|\log T = \infty. 
                 \end{array} \right.
$$
\end{conj}
Note that our conjecture specializes to K\"{o}sters's conjecture when $ \alpha_1 - \alpha_2 = c/\log T,$ for some fixed constant $c \in {\mathbb R}.$ We will discuss conjecture \ref{eqn:conjcor} in more detail in \S 2. 

Even though we cannot prove the asymptotic formula for $M_{\bf k}(T, {\overrightarrow{\alpha}}),$ assuming RH, we are able to find similar upper bound to (\ref{eqn:uppSound}) as in the following theorem.
\begin{thm}\label{thm:main} 
Assume RH. Let ${\bf k} = (k_1,...,k_m)$ be a sequence of positive real numbers and ${\overrightarrow{\alpha}} = (\alpha_1,..., \alpha_m)$ be defined as in (\ref{eqn:Mkt}). Then for $T$ large, 
$$ M_{\bf k}(T,{\overrightarrow{\alpha}}) \ll_{{\bf k},\epsilon} T(\log T)^{k_1^2 + k_2^2 +...+ k_m^2 + \epsilon}\prod_{i < j}\big({\rm min}\{\frac{1}{|\alpha_i - \alpha_j|}, \log T\}\big)^{2k_ik_j}. $$
\end{thm}
To obtain the upper bound above, we follow Soundararajan's techniques for finding upper bounds for the moments of Riemann zeta function \cite{Moment}.  Soundararajan's work is built on Selberg's work on the distribution of $\log |\zeta(\h + it)|$ \cite{Selberg}. He started from estimating an upper bound for ${\rm meas} (A(T,V))$, where $V \geq 3,$ and $A(T,V) = \{ t \in [T, 2T]: \log |\zeta(\h + it)| \geq V \}.$ Then he observed that 
$$\int_T^{2T} |\zeta(1/2 + it)|^2k \> dt = - \int_{-\infty}^{\infty} e^{2kV} \>d\,\, \textup{meas}(A(T,V)) = 2k \int_{-\infty}^{\infty} e^{2kV} \,\, \textup{meas}(A(T,V)) \> dV.$$
Hence an upper bound for the moment of the Riemann zeta function in (\ref{eqn:uppSound}) is deduced from the upper bound of $\textup{meas}(A(T,V))$.
For the shifted moments, we will instead estimate
\begin{equation} \label{eqn:stv}
 S(T,V) = \{t \in [T,2T]: \log|\zeta(\h + it + i\alpha_1)|^{k_1} +...+ \log|\zeta(\h + it + i\alpha_m)|^{k_m} \geq V \}.
\end{equation}  
The rest of our proof of Theorem \ref{thm:main} is then analogous to Soundararajan's proof of the Theorem in \cite{Moment}, except that we are required to use Lemma \ref{uppCo}. The detail of the proof will be discussed in \S 3. Recently, Soundararajan and Young \cite{SY} have used a similar version of Lemma \ref{uppCo} and similar extension of Soundararajan's work to obtain the second moment of quadratic twists of modular $L$-functions. 
\\
\\
We also establish a lower bound for $M_{\bf k}(T,{\overrightarrow{\alpha}})$ unconditionally in Theorem \ref{lowerbound} below. 
\begin{thm} \label{lowerbound}
Unconditionally, for large $T$, ${\bf k} = (k_1,...,k_m)$ a sequence of positive integers, $|\alpha_i - \alpha_j| = O(1)$ for any $i,j = 1,..,m,$ and $|\alpha_i| = O(\log \log T),$
$$ M_{{\bf k}}(T, \alpha) \gg_{{\bf k}, \beta} T (\log T)^{k_1^2 + ... + k_m^2}\prod_{i < j}\big({\rm min}\{\frac{1}{|\alpha_i - \alpha_j|}, \log T\}\big)^{2k_ik_j}, $$
where 
\begin{equation} \label{eqn:beta}
\beta := \max_{\{(i,j) | |\alpha_i - \alpha_j| = O(1/\log T)\}} \{ \lim_{T \rightarrow \infty} |\alpha_i - \alpha_j| \log T \}. 
\end{equation}
\end{thm}
The proof of Theorem \ref{lowerbound} uses similar techniques to Rudnick's and Soundararajan's work on finding lower bounds for the moments of a family of Dirichlet $L$-functions \cite{Lower}. Let
\begin{equation} \label{eqn:atintro}
A(t) = \sum_{j \leq x} \frac{\sum_{n_1n_2...n_m = j}d_{k_1}(n_1)d_{k_2}(n_2)...d_{k_m}(n_m)n_1^{-i\alpha_1}n_2^{-i\alpha_2}...n_m^{-i\alpha_m}}{j^{\frac{1}{2}+it}},
\end{equation}
and $d_k(n) = \sum_{a_1a_2...a_k = n} 1$ and $x = T^{1/2}.$ Note that $A(t)$ is a short truncation of  
$$\zeta(\h + it + i\alpha_1)^{k_1}\zeta(\h + it + i\alpha_2)^{k_2} ... \zeta(\h + it + i\alpha_m)^{k_m}.$$
We will compute a lower bound for 
\begin{equation} \label{eqn:S1}
 S_1 = \left| \int_{-\infty}^{\infty} \zeta(\h + it + i\alpha_1)^{k_1}\zeta(\h + it + i\alpha_2)^{k_2} ... \zeta(\h + it + i\alpha_m)^{k_m} \overline{A(t)} K\left(\frac{t}{T} \right) \>dt \right|  , 
\end{equation}
where $K(x)$ is a nonnegative bounded function in ${\bf C}^{\infty}(\R)$ and compactly support in [1,2].  Also we will find an upper bound for    
$$ S_2 = \int_T^{2T} |A(t)|^{2} \>dt. $$
By Cauchy-Schwarz's inequality, we obtain
\begin{eqnarray*}
&& \int_T^{2T} \left| \zeta(\h + it + i\alpha_1)^{2k_1}\zeta(\h + it + i\alpha_2)^{2k_2} ... \zeta(\h + it + i\alpha_m)^{2k_m} \right| \> dt  \\
&\gg&  \int_{-\infty}^{\infty} \left| \zeta(\h + it + i\alpha_1)^{2k_1}\zeta(\h + it + i\alpha_2)^{2k_2} ... \zeta(\h + it + i\alpha_m)^{2k_m} \right| K\left(\frac{t}{T} \right) \> dt \\
&\geq&  S_1^2/ \left( \int_{-\infty}^{\infty} |A(t)|^{2} K\left(\frac{t}{T} \right) \>dt \right) \\
&\gg& \frac{S_1^2}{S_2}.
\end{eqnarray*}
Theorem \ref{lowerbound} will follow from showing that 
$$S_1 \gg_{{\bf k}, \beta} T (\log T)^{k_1^2 + ... + k_m^2}\prod_{i < j}\big({\rm min}\{\frac{1}{|\alpha_i - \alpha_j|}, \log T\}\big)^{2k_ik_j} \,\,\,\,\,({\rm Lemma \,\, \ref{lem:S1}}), $$
and 
$$ S_2 \ll_{{\bf k}, \beta} T (\log T)^{k_1^2 + ... + k_m^2}\prod_{i < j}\big({\rm min}\{\frac{1}{|\alpha_i - \alpha_j|}, \log T\}\big)^{2k_ik_j} \,\,\,\,\, ({\rm Lemma \,\, \ref{lem:at}}). $$
\\
Again in the case where $m = 2$, we obtain from Theorem \ref{thm:main} and \ref{lowerbound} that 
$$ T(\log T)^{2k^2} \ll M_{k,k}(\alpha_1, \alpha_2) \ll  T(\log T)^{2k^2 + \epsilon} $$
when $\lim_{T \rightarrow \infty} |\alpha_1 - \alpha_2|\log T = \infty$. Furthermore
$$ T(\log T)^{4k^2} \ll M_{k,k}(\alpha_1, \alpha_2) \ll  T(\log T)^{4k^2 + \epsilon}$$
when $\lim_{T \rightarrow \infty} |\alpha_1 - \alpha_2|\log T < \infty.$ the order of the leading asymptotic term of the upper and lower bounds for $M_{k,k}(\alpha_1, \alpha_2)$ correspond to the one in our conjecture \ref{eqn:conjcor}. The result suggests that the correlation of $|\zeta(\h + it + i\alpha_1)|$ and $|\zeta(\h + it + i\alpha_2)|$ transition at $|\alpha_1 - \alpha_2| \approx \frac{1}{\log T}$. In particular these distribution appear independent when $|\alpha_1 - \alpha_2|$ is much larger than $\frac{1}{\log T}.$
\\
\\
Acknowledgements: I am very grateful to Professor Soundararajan for his guidance throughout the making of this paper. I would like to thank Xiannan Li for helpful editorial comments. I would like to thank Professor Nina Snaith for suggesting a useful reference.

\section{Conjecture for the Shifted Moments}
Based on Keating and Snaith's random matrix model, to conjecture the shifted moments $M_{(k,k)}(T, (\alpha_1, \alpha_2)),$ we need to compute asymptotic formula for 
$$ g_{(k,k)}(N, (\alpha_1, \alpha_2)) :=  \int_{U(N)} |\Lambda(e^{i\alpha_1})|^{2k}|\Lambda(e^{i\alpha_2})|^{2k} \> dU_N.$$
Clearly, $g_{(k,k)}(N, (\alpha_1, \alpha_2)) = g_{(k,k)}(N, (0,\alpha_1 - \alpha_2)).$ Using the scaling $N = \log (T/2\pi),$ we can derive conjecture \ref{eqn:conjcor} from the following proposition. 
\begin{prop} \label{prop:rmt} Let $\alpha$ be fixed functions in term of $T$ such that $\lim_{N \rightarrow \infty} \alpha N$ exists or equals $\pm \infty$ and $\alpha_1 \neq n\pi,$ where $n is integer.$ As $N \rightarrow \infty$ we obtain
$$ g_{(k,k)}(N, (0, \alpha)) \sim \left\{ \begin{array}{ll}
                     \frac{G^2(2k+1)}{G(4k+1)} N^{4k^2} & {\rm if} \lim_{N \rightarrow \infty} |\alpha|N = 0, \\
                    C_k N^{4k^2} & {\rm if} \lim_{N \rightarrow \infty} |\alpha|N = c \neq 0, \\
                    |1 - e^{i\alpha}|^{-2k^2} \frac{G^4(k+1)}{G^2(2k+1)} N^{2k^2} & {\rm if} \lim_{N \rightarrow \infty} |\alpha|N = \infty,
                 \end{array} \right.
$$
where $C_k = \lim_{\substack{\mu_1,...\mu_k \rightarrow  0 \\ \mu_{k+1},..., \mu_{2k} \rightarrow c/2\pi }}  \frac{\det (b_{jl})_{j,l = 1,...,2k}}{\Delta^2(2\pi\mu_1,...,2\pi\mu_{2k})}$, and $b_{jl}$ is defined as in (\ref{eqn:conjshift}). 
\end{prop}
The idea of the proof when $\lim_{\rightarrow} |\alpha|N = \infty$ is similar to the proof of Lemma 3 in \cite{KO}. The result when $\lim_{\rightarrow} |\alpha|N = c \neq 0$ is due to K\"{o}sters in \cite{Kshift}.
\begin{proof} First we begin the proof by the following identity, which will be useful later. 
\begin{equation} \label{eqn:lim_gU}
\lim_{N \rightarrow \infty} \frac{g_U(N,k)}{N^{k^2}} = \frac{1}{(2\pi i)^{2k}}\frac{1}{(k)!^2}\oint... \oint \frac{\prod_{i, j = 1}^{k} \frac{1}{v_{j+k} - v_i}\Delta^2(v_1, ...., v_{2k}) \exp(\tfrac{1}{2} \sum_{j=1}^k (v_j - v_{k+j}))}{\prod_{j=1}^{2k} v_j^{2k}} \>dv_1 dv_2...dv_{2k},
\end{equation}
where we integrate over small circles around $v_i = 0$. This equation is proved in Lemma 5 of \cite{KO}. Next from equation (1.5.9) of \cite{CFKRS}, we have
\begin{eqnarray*}
 && g_{(k,k)}(N, (0, \alpha)) = \int_{U(N)} |\Lambda(1)|^{2k}|\Lambda(e^{i\alpha})|^{2k} \> d U_N \\
&=& \frac{1}{(2\pi i)^{4k}} \frac{1}{(2k)!^2} \oint...\oint \frac{\prod_{i, j = 1}^{2k} (1 - e^{-z_i + z_{j + 2k}})^{-1}\Delta^2(z_1, ...., z_{4k})e^{\tfrac{N}{2}\sum_{j=1 }^{2k}(z_j - z_{j+2k})}}{\prod_{j = 1}^{4k}z_j^{2k} (z_j - i\alpha)^{2k}} \> dz_1 ...dz_{4k},  
\end{eqnarray*}
where the path of integration encloses $i\alpha,$ and 0. 
\\
\\
{\bf Case 1:} $\lim_{N \rightarrow \infty} |\alpha|N = 0.$ 

Let $z_j = v_j/N.$ The integral above becomes
$$\frac{1}{(2\pi i)^{4k}} \frac{1}{(2k)!^2} \oint...\oint \frac{\prod_{i, j = 1}^{2k} (1 - e^{\tfrac{-v_i + v_{j + 2k}}{N}})^{-1}\Delta^2(v_1, ...., v_{4k})e^{\tfrac{1}{2}\sum_{j=1 }^{2k}(v_j - v_{j+2k})}}{\prod_{j = 1}^{4k}v_j^{2k} (v_j - iN\alpha)^{2k}} \> dv_1 ...dv_{4k}. $$
As $N \rightarrow \infty,$ $(1 - e^{\tfrac{-v_i + v_{j + 2k}}{N}})^{-1} \sim \frac{N}{v_i - v_{j + 2k}},$ and $N\alpha \sim 0.$ Therefore 
\begin{eqnarray*}
&& \lim_{N \rightarrow \infty} \frac{g_{(k,k)}(N, (0, \alpha))}{N^{4k^2}} \\
&=& \frac{1}{(2\pi i)^{4k}} \frac{1}{(2k)!^2} \oint...\oint \frac{\prod_{i, j = 1}^{2k} \frac{1}{v_i - v_{j + 2k}} \Delta^2(v_1, ...., v_{4k})e^{\tfrac{1}{2}\sum_{j=1 }^{2k}(v_j - v_{j+2k})}}{\prod_{j = 1}^{4k}v_j^{4k}} \> dv_1 ...dv_{4k} \\
&=& \lim_{N \rightarrow \infty} \frac{g_U(N,2k)}{N^{4k^2}} = \frac{G^2(2k+1)}{G(4k+1)} , 
\end{eqnarray*}
where the last line follows from (\ref{eqn:moment_rmt}) and (\ref{eqn:lim_gU}).
\\
\\
{\bf Case 2:} $\lim_{N \rightarrow \infty} |\alpha|N = \infty.$

For this case, each contour can be deformed to two small circular contours centered respectively at the poles 0, $i\alpha.$ They are connected by two straight line paths which cancel each other. Therefore we can consider the contour integral above as a sum of $2^{4k}$ integrals in which each $z_j$ runs over one of the smaller circular paths. Let $\epsilon \in \{0,1\},$ and $\gamma_{\epsilon_j}$ be a circle with center $\epsilon_j i \alpha$ and small radius (less than $|\alpha|/2N$). Let 
\begin{eqnarray*}
&& I(N, k, \epsilon_1,..., \epsilon_{4k}) \\
&=& \frac{1}{(2\pi i)^{4k}} \frac{1}{(2k)!^2} \oint_{\gamma_{\epsilon_1}}...\oint_{\gamma_{\epsilon_{4k}}} \frac{\prod_{i, j = 1}^{2k} (1 - e^{-z_i + z_{j + 2k}})^{-1}\Delta^2(z_1, ...., z_{4k})e^{\tfrac{N}{2}\sum_{j=1 }^{2k}(z_j - z_{j+2k})}}{\prod_{j = 1}^{4k}z_j^{2k} (z_j - i\alpha)^{2k}} \> dz_1 ...dz_{4k}.
\end{eqnarray*} 
Hence
$$ \int_{U(N)} |\Lambda(1)|^{2k}|\Lambda(e^{i\alpha})|^{2k} \> d U_N = \sum_{\epsilon_j \in \{0, 1\}} I(N, k, \epsilon_1,..., \epsilon_{4k}). $$
Now we consider $I(N, k, \epsilon_1,..., \epsilon_{4k}).$ We change varibles $z_j = v_j/N + i\epsilon_j \alpha$ and obtain
\begin{eqnarray*}
&& I(N, k, \epsilon_1,..., \epsilon_{4k}) \\
&=& \frac{1}{(2\pi i)^{4k}} \frac{1}{(2k)!^2} \oint_{\gamma_0}...\oint_{\gamma_0} 
\prod_{i, j = 1}^{2k} (1 - e^{\tfrac{-v_i + v_{j + 2k}}{N} + i(\epsilon_{j+2k} - \epsilon_i)\alpha})^{-1} \Delta^2(\frac{v_1}{N} + i\epsilon_1\alpha, ...., \frac{v_{4k}}{N} + i\epsilon_{4k}\alpha) \\
&& \frac{\exp(\tfrac{1}{2}\sum_{j=1}^{2k}(v_j - v_{j+2k}) + iN(\epsilon_j - \epsilon_{j+2k})\alpha)}{N^{4k} \prod_{j = 1}^{4k}(\tfrac{v_j}{N} + i\epsilon_j \alpha)^{2k} (\tfrac{v_j}{N} + i\epsilon_j \alpha - i\alpha)^{2k}} \> dv_1 ...dv_{4k}. 
\end{eqnarray*}
For large $N$, $(1 - e^{\tfrac{-v_i + v_{j + 2k}}{N}})^{-1} \sim \tfrac{N}{v_i - v_{j+ 2k}}.$ Since $\lim_{N \rightarrow \infty} |\alpha| N = \infty,$ as $N \rightarrow \infty,$  we have
\begin{eqnarray} \label{eqn:INkepsilon}
&& I(N, k, \epsilon_1,..., \epsilon_{4k}) \nonumber \\
&\sim& \frac{1}{(2\pi i)^{4k}} \frac{N^{8k^2 - 4k}}{(2k)!^2} \oint_{\gamma_0}...\oint_{\gamma_0} \prod_{\substack{i,j = 1 \\ \epsilon_{j+2k} = \epsilon_i}}^{2k} \tfrac{N}{v_{j+2k} - v_i}\prod_{\substack{i,j = 1 \nonumber \\ \epsilon_{j+2k} \neq \epsilon_i}}^{2k} (1 - e^{i(\epsilon_{j+2k} - \epsilon_i)\alpha})^{-1} \prod_{\substack{i < j \\ \epsilon_i = \epsilon_j}} \left(\tfrac{v_j - v_i}{N} \right)^2 \prod_{\substack{i < j \\ \epsilon_i \neq \epsilon_j}} (i\alpha)^{2} \nonumber \\
&& \frac{\exp(\tfrac{1}{2} \sum_{j=1}^{2k} (v_i - v_{j+2k}) + i\tfrac{\alpha N}{2} \sum_{\epsilon_{j+2k} \neq {\epsilon_j}} (\epsilon_j - \epsilon_{j+2k}))}{\prod_{j=1}^{4k} v_j^{2k} (i\alpha)^{2k}} \> dv_1 ...dv_{4k}. \nonumber
\end{eqnarray} 
As $N \rightarrow \infty,$ we claim that the main contribution is from terms $I(N,k, \epsilon_1, ..., \epsilon_{4k})$ such that both the number of $1's$ among $\epsilon_1,...\epsilon_{2k}$ and that among $\epsilon_{2k+1},...\epsilon_{4k}$ equal $k$. This claim will be proved at the end of the proof of Case 2. There are ${2k \choose k}^2$ such terms. By symmetry, $I(N, k, \epsilon_1,...\epsilon_{4k})$ with properties above are all identical. In fact the integrand in the equation above is equal to 
\begin{eqnarray*}
&& |1 - e^{i\alpha}|^{-2k^2}  \prod_{\substack{i,j = 1 \\ \epsilon_{j+2k} = \epsilon_i = 1} }^{2k} \tfrac{N}{v_{j+2k} - v_i} \prod_{\substack{i,j = 1 \\ \epsilon_{j+2k} = \epsilon_i = 0} }^{2k} \tfrac{N}{v_{j+2k} - v_i}  \\
&& \cdot \prod_{\substack{i < j \\ \epsilon_i = \epsilon_j = 1}} \left(\tfrac{v_j - v_i}{N} \right)^2 \prod_{\substack{i < j \\ \epsilon_i = \epsilon_j = 0}} \left(\tfrac{v_j - v_i}{N} \right)^2 \tfrac{\exp(\tfrac{1}{2} \sum_{j = 1}^{2k} (v_j - v_{j+2k}))}{\prod_{j = 1}^{4k} v_j^{2k}}
\end{eqnarray*}
Therefore as $N \rightarrow \infty,$
\begin{eqnarray*}
 && \int_{U(N)} |\Lambda(1)|^{2k}|\Lambda(e^{i\alpha})|^{2k} \> d U_N \sim {2k \choose k}^2  I(N,k, \underbrace{0,..,0}_{\# 0's = k}, \underbrace{1,...,1}_{\#1's = k},\underbrace{0,..,0}_{\# 0's = k}, \underbrace{1,...,1}_{\#1's = k}) \\
&\sim& {2k \choose k}^2  \frac{1}{(2\pi i)^{4k}} \frac{1}{(2k)!^2} |1 - e^{i\alpha}|^{-2k^2} \\
&& \cdot \left( \oint_{\gamma_0}... \oint_{\gamma_0} \frac{\prod_{i, j = 1}^{k} \frac{N}{v_{j+k} - v_i} \Delta^2(v_1, ...., v_{2k})\exp(\tfrac{1}{2} \sum_{j=1}^k (v_j - v_{k+j}))}{\prod_{j=1}^{2k} v_j^{2k}} \> dv_1...dv_{2k}\right)^2\\
&\sim& |1 - e^{i\alpha}|^{-2k^2} g_U^2(N,k) \sim |1 - e^{i\alpha}|^{-2k^2} \frac{G^4(k+1)}{G^2(2k+1)} N^{2k^2}.
\end{eqnarray*}
To complete the proof of Case 2, we will prove the claim above. Let $l_1$ and $m_1$ be the number of 1's among $\epsilon_1,..., \epsilon_{2k}$ and $\epsilon_{2k + 1},..., \epsilon_{4k}$ respectively. From the equation before the claim, the leading order term of $I(N, k, \epsilon_1,..., \epsilon_{4k})$ is
\begin{eqnarray*}
&& \frac{1}{N^{4k^2 - 6kl_1 - 6km_1 + 2l_1^2 + 2m_1^2 + 2l_1m_1}}\frac{1}{|\alpha|^{8k^2 - 2(l_1 + m_1)(4k - l_1 - m_1)}}\cdot |1 - e^{i\alpha}|^{-m_1l_2}|1 - e^{-i\alpha}|^{-m_2l_1} \\
&\sim& \frac{1}{(N|\alpha|)^{4k^2 - 6kl_1 - 6km_1 + 2l_1^2 + 2m_1^2 + 2l_1m_1}}\frac{1}{|\alpha|^{4k^2}}.
\end{eqnarray*}
The main contribution comes from terms such that $4k^2 - 6kl_1 - 6km_1 + 2l_1^2 + 2m_1^2 + 2l_1m_1$ is minimum, where $0 \leq l_1, k_1 \leq 2k$. Now let $t = l_1 + m_1.$ Hence for $0 \leq l_1 \leq t$ and $0 \leq t \leq 4k,$ we have 
$$4k^2 - 6kl_1 - 6km_1 + 2l_1^2 + 2m_1^2 + 2l_1m_1 = 4k^2 - 6kt + 2t^2 - 2l_1t + 2l_1^2.$$
By calculus, the minimum value of the above is $-2k^2$, which occurs when $t = 2k$ and $l_1 = t/2 = k.$ This proves the claim. 
\\
\\
{\bf Case 3:} $\lim_{N \rightarrow \infty} |\alpha| N = c \neq 0.$ 

By the same arguments as case 1, we obtain 
$$\lim_{N \rightarrow \infty} \frac{g_U(N, (0,\alpha))}{N^{4k^2}} = \lim_{N \rightarrow \infty}\frac{ g_U(N, (0, c))}{N^{4k^2}}. $$
From equation (1.2) of \cite{Kshift}, K\"{o}sters showed that 
$$
\lim_{N \rightarrow \infty} \frac{1}{N^{4k^2}} \int_{U(N)} \prod_{\mu_1,..,\mu_{2k}} |\Lambda(e^{2\pi \mu_i})|^{2} \> dU_N =
\frac{1}{\Delta(2\pi\mu_1,..., 2\pi\mu_2)} \det(b_{jl})_{j,l = 1,..., 2k},
$$
where $b_{jl}$ is defined as in (\ref{eqn:conjshift}). From two equations above, we then have  
\begin{equation} \label{eqn:koster_moment}
 \lim_{N \rightarrow \infty} \frac{g_U(N, (0, \alpha))}{N^{4k^2}} = C_k := \lim_{\substack{\mu_1,...\mu_k \rightarrow  0 \\ \mu_{k+1},..., \mu_{2k} \rightarrow c/2\pi }}  \frac{\det (b_{jl})_{j,l = 1,...,2k}}{\Delta^2(2\pi\mu_1,...,2\pi\mu_{2k})}. 
\end{equation}
$C_k$ exists by continuous extension (we will prove this in Appendix 5.1). This concludes the proof of the proposition. 
\end{proof}

\section{Proof of Theorem \ref{thm:main}}
Let $S(T,V)$ be defined as in (\ref{eqn:stv}) and observe that 
\begin{eqnarray}\label{zetaStv}
&& \int_T^{2T} |\zeta(\h + it + i\alpha_1)|^{2k_1}|\zeta(\h + it + i\alpha_2)|^{2k_2}... |\zeta(\h + it + i\alpha_m)|^{2k_m} \> dt \\
&=& - \int_{-\infty}^{\infty} e^{2V} \>d\,\, \textup{meas}(S(T,V))
= 2\int_{-\infty}^{\infty} e^{2V} \textup{meas}(S(T,V)) \>dV. \nonumber
\end{eqnarray}
To prove the upper bound in Theorem \ref{thm:main}, we need to estimate the measure of $S(T,V)$ for large $T$ and all $V \geq 3$. Throughout this section, we will let 
$$W = (k_1^2 + ... + k_m^2) \log \log T + \sum_{\substack{i,j \\ i < j}} 2k_ik_j\log(\,{\rm min}\, (\frac{1}{|\alpha_i - \alpha_j|}, \log T)).$$
\begin{thm}\label{thm:measS}
Assume RH. Let $T$ be large and $V \geq 3$ be a real number. If $10\sqrt{\log \log T} \leq V \leq W$ then
$$ \textup{meas}(S(T,V)) \ll T \frac{V}{\sqrt{W}} \exp\left(-\frac{V^2}{W}\left(1-\frac{4}{\log W} \right)\right); $$
if \,\, $W < V \leq \h W\log W $ we have
$$ \textup{meas}(S(T,V)) \ll T \frac{V}{\sqrt{W} } \exp\left(-\frac{V^2}{W}\left(1-\frac{7V}{4 W \log W} \right)^2\right); $$
and if \,\, $\h W\log W < V$ we have
$$ \textup{meas}(S(T,V)) \ll T\exp\left(-\frac{1}{129}V\log V\right). $$
\end{thm}
The upper bound in Theorem \ref{thm:main} follows from inserting the upper bounds of Theorem \ref{thm:measS} in equation (\ref{zetaStv}). In fact, we need only the crude upper bound:
$$
\textup{meas}(S(T,V)) \ll \left\{ \begin{array}{ll}
                   T(\log T)^{o(1)}\exp\left(-\frac{V^2}{W}\right) & \,\,\,\,{\rm \ if\ } 3 \leq V \leq 256 W, \\
                   T(\log T)^{o(1)}\exp(-4V)  & \,\,\,\, {\rm \ if\ } 256 W < V.  
                 \end{array} \right.
$$
As mentioned in the introduction, the proof of the theorem is similar to the proof of the theorem in \cite{Moment}. We will exploit the following proposition and lemmas below, the proof of which can be found in \cite{Moment}.
\begin{prop} \label{prop:main}
Assume RH. Let T be large, $t \in [T,2T]$, and $2 \leq x \leq T^2.$ Let $\lambda_0 = 0.4912...$ denote the unique positive real number satisfying $e^{-\lambda_0} = \lambda_0 + \lambda^2_0/2.$ For $\lambda \geq \lambda_0$ we have the estimate
$$ \log|\zeta(\tfrac{1}{2} + it)| \leq \R \sum_{n \leq x} \frac{\Lambda(n)}{n^{\h + \frac{\lambda}{\log x} +it}\log n}\frac{\log(x/n)}{\log x} + \frac{(1+ \lambda)}{2} \frac{\log T}{\log x} + O(\frac{1}{\log x}).$$
\end{prop}
\begin{lemma}\label{Lem:2}
Assume RH. Let $T \leq t \leq 2T, 2 \leq x \leq T^2,$ and let $\sigma \geq \h$. Then 
$$ \left|\sum_{\substack{n\leq x \\ n \neq p}} \frac{\Lambda(n)}{n^{\sigma + it}\log n }\frac{\log(x/n)}{\log x} \right| \ll \log \log \log T + O(1). $$  
\end{lemma} 
\begin{lemma}\label{Lem:3}
Let $T \leq t \leq 2T,$ and $2 \leq x \leq T^2.$ Let $k$ be a natural number such that $x^k \leq T/\log T.$ For any complex number $a(p)$ we have
$$ \int_T^{2T} \left|\sum_{p \leq x} \frac{a(p)}{p^{\h + it}}\right|^{2k} \> dt \ll Tk!\left(\sum_{p \leq x} \frac{|a(p)|^2}{p}  \right)^k. $$
\end{lemma}   
Choosing $\lambda = 0.5$ in Proposition \ref{prop:main}, we obtain that 
\begin{eqnarray} \label{ineqlogZeta}
&&\log|\zeta(\h + it + i\alpha_1)|^{k_1} +...+ \log|\zeta(\frac{1}{2} + it + i\alpha_m)|^{k_m}\\
&\leq& \R \sum_{n \leq x} \left(\frac{k_1\Lambda(n)}{n^{\h + \frac{0.5}{\log x} +it + i\alpha_1}\log n} + ... + \frac{k_m\Lambda(n)}{n^{\h + \frac{0.5}{\log x} +it + i\alpha_m}\log n}\right)\frac{\log(x/n)}{\log x}  \nonumber \\ 
&& + \frac{3(k_1 +...+k_m)}{4} \frac{\log T}{\log x} + O\left(\frac{1}{\log x}\right).\nonumber
\end{eqnarray}
The contribution of prime powers $n = p^k,$ where $k \geq 2,$ to our sums above is negligible by Lemma \ref{Lem:2} and the triangle inequality.  Therefore to finish the proof of Theorem \ref{thm:measS}, we need to bound the sum involving primes as the following. 

From Corollary C of \cite{Moment}, assuming RH, for all large $t$ we obtain
$$|\zeta(\h + it)| \leq \exp\left( \frac{3}{8} \frac{\log t}{\log \log t}\right). $$
Therefore to prove Theorem \ref{thm:measS}, we can assume that $10 \sqrt{\log\log T} \leq V \leq \frac{3(k_1 + ... + k_m)}{8} \frac{\log T}{\log \log T}$ since $T$ will be large. We define $A$ as
$$
  A = \left\{ \begin{array}{ll}
                   \frac{(k_1 + ... + k_m)}{2} \log W  & {\rm \ if\ } 10\sqrt{\log \log T} \leq V \leq W, \\
                   \frac{(k_1 + ... + k_m)}{2V}W \log W & {\rm \ if\ } W < V \leq \frac{1}{2}W \log W, \\
                   k_1 + ... + k_m &  {\rm \ if\ } V > \frac{1}{2}W \log W.
                 \end{array} \right.
$$
Let $x = T^{A/V}$ and $z = x^{1/\log \log T}.$ By Lemma \ref{Lem:2} and inequality (\ref{ineqlogZeta}), we have
\begin{eqnarray*}
&& \log|\zeta(\h + it + i\alpha_1)|^{k_1} +...+ \log|\zeta(\h + it + i\alpha_m)|^{k_m} \\
&\leq& S_1(t) + S_2(t) + \frac{3(k_1 + ....+ k_m)}{4}\frac{V}{A} + O(\log \log \log T), 
\end{eqnarray*}
where 
$$ S_1(t) = \left|\sum_{p \leq z} \frac{(k_1p^{-i\alpha_1} + ... + k_mp^{-i\alpha_m})}{p^{\h+\frac{0.5}{\log x}+ it}} \frac{\log\frac{x}{p}}{\log x} \right|   $$ 
and 
$$ S_2(t) = \left|\sum_{z < p \leq x} \frac{(k_1p^{-i\alpha_1} + ... + k_mp^{-i\alpha_m})}{p^{\h+\frac{0.5}{\log x}+ it}} \frac{\log\frac{x}{p}}{\log x} \right|.  $$
If $t \in S(T,V)$ then we must either have
$$ S_1(t) \geq V_1 := V\left(1 - \frac{7(k_1 +...+ k_m)}{8A}\right)  \,\,\,\,\,\,\, \textup{or} \,\,\,\,\,\,\, S_2(t) \geq \frac{(k_1 + ... +k_m)V}{8A}.  $$
Let 
$$\textup{meas}(S_1) := \{t \in [T,2T] : S_1(t) \geq  V_1 \}, $$
and 
$$\textup{meas}(S_2) := \{t \in [T,2T] : S_2(t) \geq \frac{(k_1 + ... +k_m)V}{8A} \}. $$
The proof of Theorem \ref{thm:measS} will follow easily from Lemma \ref{lem:measureofs1s2} below, which is to find upper bounds for $\textup{meas}(S_1)$ and $\textup{meas}(S_2).$ In order to do this, Lemma \ref{uppCo} below is crucial for obtaining the upper bounds for $\textup{meas}(S_1)$ in Lemma \ref{lem:measureofs1s2}.
\begin{lemma} \label{uppCo} For $|a| = O(1),$
$$\sum_{p \leq z} \frac{\cos (a\log p)}{p} \leq \log\big({\rm min} \{\frac{1}{|a|}, \log z\}\big) + O(1).$$
\end{lemma} 
\begin{proof} We can assume that $a$ is positive. From Theorem 2.7 of \cite{MV}, 
$$ \sum_{p \leq x} \frac{1}{p} = \log \log x + C + h(x); \,\,\,\, h(x) = O\left( \frac{1}{\log x} \right). $$
If $z \leq e^{1/a}$, then $\log z \leq \frac{1}{a}$ and 
$$ \sum_{p \leq z} \frac{\cos (a\log p)}{p} \leq \sum_{p \leq z} \frac{1}{p} = \log \log z + O(1). $$
Otherwise, we can write
$$\sum_{p \leq z} \frac{\cos (a\log p)}{p} = \sum_{p \leq e^{1/a}} \frac{\cos (a\log p)}{p} + \sum_{ e^{1/a} < p \leq z} \frac{\cos (a\log p)}{p}. $$
Since $|a| = O(1),$ the first sum is 
$$  \sum_{p \leq e^{1/a}} \frac{1 + O((a\log p)^2)}{p} = \log{\frac{1}{a}} + O(1)$$
By partial summation, the second sum is 
$$  \int_{e^{1/a}}^z \cos (a\log x) \>dg(x) $$
where 
$$ g(x) = \sum_{p \leq x} \frac{1}{p}. $$
Then we have 
$$  \int_{e^{1/a}}^z  \cos (a\log x) \>d(\log \log x) = \int_{e^{1/a}}^z \frac{\cos (a \log x)}{x\log x}\>dx = \int_{1}^{a\log z} \frac{\cos t}{t} \>dt = O(1), $$
and since $h(x) = O\big(\frac{1}{\log x}\big),$
$$ \int_{e^{1/a}}^z  \cos (a\log x) \>dh(x) = O(1). $$
Combining the inequalities above, we obtain the lemma. 
\end{proof}
\begin{lemma} \label{lem:measureofs1s2} Assume RH. Let $x$,$z$, $A$, and $V_1$  be defined as above. We have
$$ \textup{meas}(S_1) \ll T\frac{V}{\sqrt{W}}\exp\left(-\frac{V_1^2}{W} \right) + T\exp(-4V\log V), $$
and 
$$\textup{meas}(S_2) \ll  T\exp\left(-(k_1 + ... + k_m)\frac{V}{2A}\log V \right).$$ 
\end{lemma}
\begin{proof}
By Lemma \ref{Lem:3}, for any natural number $k \leq (k_1 + ... + k_m)V/A - 1$, we obtain
$$ \int_T^{2T} |S_2(t)|^{2k} \> dt \ll Tk!\left( \sum_{z < p \leq x} \frac{|k_1p^{-i\alpha_1} +...+ k_mp^{-i\alpha_m}|^2}{p} \right)^k \ll T(k(k_1 + ... + k_m)^2(\log_3T + O(1)))^k.  $$
Choosing $k = \lfloor (k_1 + ... + k_m)V/A-1 \rfloor$, we derive that 
$$ \textup{meas}(S_2) \ll T\left(\frac{8A}{V}\right)^{2k}(2k\log_3T)^k \ll T\exp\left(-(k_1 + ... + k_m)\frac{V}{2A}\log V \right). $$ 
Next we find the upper bound of $\textup{meas}(S_1)$. By Lemma \ref{Lem:3}, for any $k \leq \log(T/\log T)/\log z,$
\begin{eqnarray*}
\int_{T}^{2T} |S_1(t)|^{2k} \> dt &\ll& Tk!\left(\sum_{p \leq z} \frac{|k_1p^{-i\alpha_1} +...+ k_mp^{-i\alpha_m}|^2 }{p} \right)^{k} \\
&\ll& Tk!\left(\sum_{p \leq z} \frac{k_1^2 + ...k_m^2 + 2 \sum_{i < j} k_ik_j\cos((\alpha_i - \alpha_j)\log p) }{p} \right)^{k}
\end{eqnarray*}
From Lemma \ref{uppCo} and Stirling's formula, for $|\alpha_i - \alpha_j| = O(1),$ we obtain
$$\int_T^{2T} |S_1(t)|^{2k} \> dt \ll  T\sqrt{k}\left(\frac{kW}{e}\right)^k. $$
Hence 
$$ \textup{meas}(S_1) \ll T\sqrt{k}\left(\frac{kW}{eV_1^2}\right)^k. $$
When $V \leq W^2,$ choose $k = \lfloor \frac{V_1^2}{W} \rfloor,$ and when $V >  W^2,$ choose $k = \lfloor 10V \rfloor.$ Hence 
$$\textup{meas}(S_1) \ll T\frac{V}{\sqrt{W}}\exp\left(-\frac{V_1^2}{W} \right) + T\exp(-4V\log V). $$
\end{proof}
This concludes the proof of Theorem \ref{thm:measS}.
\section{Proof of Theorem \ref{lowerbound}}
Let $A(t)$ and $M_{{\bf k}}(T, \overrightarrow{\alpha})$ be defined as in (\ref{eqn:Mkt}) and (\ref{eqn:atintro}). Here we add extra conditions, which are that $|\alpha_i| = O(\log \log T),$ and ${\bf k} = (k_1,...,k_m)$ is a sequence of positive integers. As discussed in the introduction, the proof of  the theorem follows from Lemma \ref{lem:at} and \ref{lem:S1}. Throughout this section, $\beta$ is defined as in (\ref{eqn:beta})

\begin{lemma} \label{lem:at} Unconditionally for large $T$
$$ \int_0^T |A(t)|^{2} \>dt \sim T (\log T)^{k_1^2 + ... + k_m^2}\prod_{i < j}\big({\rm min}\{\frac{1}{|\alpha_i - \alpha_j|}, \log T\}\big)^{2k_ik_j}, $$
where the imiplied constant depends on ${\bf k} = (k_1,...,k_m)$ and $\beta.$
\end{lemma}
\begin{proof}
By dyadic summation, it suffices to prove that 
$$ \int_T^{2T} |A(t)|^{2} \>dt \sim T (\log T)^{k_1^2 + ... + k_m^2}\prod_{i < j}\big({\rm min}\{\frac{1}{|\alpha_i - \alpha_j|}, \log T\}\big)^{2k_ik_j}. $$

From the definition of $A(t),$ we obtain that for $x = \sqrt{T},$ 
$$ \lim_{T \rightarrow \infty} \frac{1}{T} \int_T^{2T} |A(t)|^{2} \>dt =  \sum_{j \leq x} \frac{|\sum_{n_1n_2...n_m = j}d_{k_1}(n_1)d_{k_2}(n_2)...d_{k_m}(n_m)n_1^{-i\alpha_1}n_2^{-i\alpha_2}...n_m^{-i\alpha_m}|^2}{j}. $$
Throughout the proof of Lemma \ref{lem:at}, we will let $x = \sqrt{T}.$ 
\\
\\
Let $$D(j) = |\sum_{n_1n_2...n_m = j}d_{k_1}(n_1)d_{k_2}(n_2)...d_{k_m}(n_m)n_1^{-i\alpha_1}n_2^{-i\alpha_2}...n_m^{-i\alpha_m}|^2,$$ and $H(s) = \sum_{j} \frac{D(j)}{j^{1+s}}.$ Note that $D(j)$ is a multiplicative function. Therefore, for $ c > 0$
$$ \int_{(c)}  \frac{H(s)x^s}{s} \>ds = \sum_{j \leq x} \frac{|\sum_{n_1n_2...n_m = j}d_{k_1}(n_1)d_{k_2}(n_2)...d_{k_m}(n_m)n_1^{-i\alpha_1}n_2^{-i\alpha_2}...n_m^{-i\alpha_m}|^2}{j}. $$
For $ \Re(s) > 0, H(s)$ converges and 
\begin{eqnarray*}
H(s) &=& \prod_{p} \left( 1 + \frac{|k_1p^{-i\alpha_1} + k_2p^{-i\alpha_2} + ... + k_mp^{-i\alpha_m}|^2}{p^{s+1}} + O\left(\frac{1}{p^{2s+2}}\right)\right)\\
     &=& \prod_{p} \left( 1 + \frac{k_1^2 +...+ k_m^2}{p^{s+1}} + \frac{\sum_{i < j} k_ik_j(p^{i(\alpha_i - \alpha_j)}+p^{-i(\alpha_i - \alpha_j)})}{p^{s+1}} + O\left(\frac{1}{p^{2s+2}}\right)\right) \\
&=& \zeta^{k_1^2 + ... + k_m^2}(s+1) \prod_{i < j} \zeta^{k_ik_j}(s+1+i(\alpha_i - \alpha_j)) \zeta^{k_ik_j}(s+1-i(\alpha_i - \alpha_j)) G(s), 
\end{eqnarray*}
where $G(s) = \prod_{p} \left(1 + O\left( \tfrac{1}{p^{2s+2}} \right) \right)$ is absolutely convergent when $\Re s > -\h.$ Therefore for $a = \frac{1}{\log x},$ 
$$ \sum_{j \leq x} \frac{D(j)}{j} = \frac{1}{2\pi i}\int_{(a)}\zeta^{k_1^2 + ... + k_m^2}(s+1) \prod_{i < j} \zeta^{k_ik_j}(s+1+i(\alpha_i - \alpha_j)) \zeta^{k_ik_j}(s+1-i(\alpha_i - \alpha_j)) G(s) \frac{x^s}{s} \>ds. $$
By corollary 5.3 in \cite{MV}, we obtain
\begin{equation} \label{truncError}
 \left| \frac{1}{2\pi i}\left(\int_{a - iY}^{a+iY} - \int_{(a)} \right) H(s)\frac{x^s}{s} \> ds \right| \ll \sum_{\h x < n < 2x} \frac{|D(n)|}{n} \min\left(1, \frac{x}{Y|x-n|} \right) + \frac{x^a}{Y} \sum_{n} \frac{|D(n)|}{n^{1+a}}. 
\end{equation}
\\
Let $R = k_1 + ... + k_m$ and ${\mathcal E} = \{ n: |n-x| \leq x/(\log x)^{2R^2}\}.$ Take $Y = \exp(\sqrt{\log x}),$ and note that $|D(n)| \leq d^2_{R}(n).$ Also for any positive interger $l$,
$$ \sum_{n \leq x} \frac{d_l^2(n)}{n} = C(\log x)^{l^2} + o((\log x)^{\l^2}),$$
so the contribution to the first sum in (\ref{truncError}) when $n \in {\mathcal E}$ is $\ll \frac{1}{(\log x)^{R^2 + 1}}.$ When $n \notin {\mathcal E}$ the contribution is 
$$\ll \left(\sum_{j \leq x} \frac{d_R^2(j)}{j} \right) \frac{(\log x)^{2R^2}}{Y} 
\ll \frac{(\log x)^{3R^2}}{Y} \ll \frac{1}{(\log x)^{R^2 + 1}}.$$
The second sum in (\ref{truncError}) is $$\ll \frac{1}{Y}  \sum_j \frac{d_R^2(j)}{j^{1+a}} \ll \frac{1}{Y} \zeta^{R^2}(1+a) \ll \frac{1}{(\log x)^{R^2 + 1}}. $$
Therefore the error from truncating the integral to $a - iY$ and $a + iY$ is $\ll \frac{1}{(\log x)^{R^2 + 1}}.$ 
\\
\\
Next let $b = - \frac{c}{\log Y},$ where $c$ is a small positive constant ($c$ is chosen to be smaller than zero-free region constant). Furthermore let $\gamma$ be the rectangle with vertices $a - iY, a + iY, b + iY, b -iY.$ Inside the rectangle $\gamma$, $H(s)\frac{x^s}{s}$ has a pole of order $k_1^2 +...+ k_m^2 + 1$ at $s = 0$ and poles at $s = i(\alpha_i \pm \alpha_j).$ Note that since we assume $|\alpha_i - \alpha_j| = O(1),$ $|\alpha_i \pm \alpha_j| \leq Y.$  Hence Cauchy's theorem gives 
\begin{eqnarray*}
\int_{a-iY}^{a+iY} + \int_{{\rm the \,\, other \,\, three \,\, sides \,\, of} \,\, \gamma}H(s)\frac{x^s}{s} \> ds  &=& {\rm res}_{s=0} H(s)\frac{x^s}{s} + \sum_{i \neq j} {\rm res}_{s = i(\alpha_i - \alpha_j)} H(s)\frac{x^s}{s}.
\end{eqnarray*}
By theorem 6.7 in \cite{MV} we obtain that on the other three sides of $\gamma,$ $\zeta^{k_1^2 + ... + k_m^2}(s + 1) \ll (\log x)^{k_1^2 + ... + k_m^2},$ and $\zeta^{k_ik_j}(s + 1 \pm i(\alpha_i - \alpha_j)) \ll (\log x)^{k_ik_j}.$ Therefore $H(s) \ll (\log x)^{R^2},$ and the contribution of the integral over the other three sides of $\gamma$ is $\ll \frac{1}{(\log x)^{R^2 + 1}}.$ 
\\
\\
Finally we need to show that 
\begin{equation} \label{eqn:resofhs}
{\rm res}_{s=0} H(s)\frac{x^s}{s} + \sum_{i \neq j} {\rm res}_{s = i(\alpha_i - \alpha_j)}H(s)\frac{x^s}{s} \sim_{{\bf k}, \beta} (\log x)^{k_1^2 + ... + k_m^2}\prod_{i < j}\big({\rm min}\{\frac{1}{|\alpha_i - \alpha_j|}, \log x\}\big)^{2k_ik_j}.
\end{equation}
Let $W := \{(i,j) \,\, |\,\, \lim_{T \rightarrow \infty} |\alpha_i - \alpha_j|\log T \,\,< \,\, \infty \,\,\, {\rm and} \,\,\, i \neq j\},$ and \\
$\widetilde{W} := \{(i,j) \,\, |\,\, \lim_{T \rightarrow \infty} |\alpha_i - \alpha_j|\log T \,\,= \,\, \infty \,\,\, {\rm and} \,\,\, i \neq j \}.$  We claim that for any $(i,j) \in \widetilde{W},$
\begin{equation} \label{eqn:negres}
  {\rm res}_{s = i(\alpha_i - \alpha_j)}H(s)\frac{x^s}{s} = o((\log x)^{k_1^2 + ... + k_m^2}\prod_{i < j}\big({\rm min}\{\frac{1}{|\alpha_i - \alpha_j|}, \log x\}\big)^{2k_ik_j}).
\end{equation}
We will provide technical details for the proof of (\ref{eqn:negres}) in Appendix 5.2 

Now to prove (\ref{eqn:resofhs}), we need to show that 
\begin{eqnarray} \label{eqn:mainres}
&& {\rm res}_{s=0} H(s)\frac{x^s}{s} + \sum_{(i,j) \in W} {\rm res}_{s = i(\alpha_i - \alpha_j)}H(s)\frac{x^s}{s} \\ \nonumber
&\sim_{{\bf k}, \beta}& (\log x)^{k_1^2 + ... + k_m^2}\prod_{i < j}\big({\rm min}\{\frac{1}{|\alpha_i - \alpha_j|}, \log x\}\big)^{2k_ik_j} \\ \nonumber
&=& (\log x)^{k_1^2 + ... + k_m^2 + \sum_{(i,j) \in W} k_ik_j} \prod_{(i,j) \in \widetilde{W}} \frac{1}{|\alpha_i - \alpha_j|^{k_ik_j}}. \nonumber
\end{eqnarray}
By Cauchy's theorem, 
$$ 
{\rm res}_{s=0} H(s)\frac{x^s}{s} + \sum_{(i,j) \in W} {\rm res}_{s = i(\alpha_i - \alpha_j)}H(s)\frac{x^s}{s} = \int_{C'} H(s)\frac{x^s}{s} \> ds,
$$
where we integrate over a circle $C'$ centered at 0 with radius $c/\log T,$ where for sufficiently large $T$, $c > \beta + 1$ ($\beta$ is defined in (\ref{eqn:beta})) for any $(i,j) \in W.$
For s on $C'$ and large $T$, $|s| < 1$, and $|s \pm i(\alpha_i - \alpha_j)| < 1$ for any $(i,j) \in W.$ By Corollary 1.6 and 1.7 in \cite{MV}, we have
\begin{eqnarray*}
&& H(s)\frac{x^s}{s} \\
&=& \left( \frac{1}{s} + \sum_{n = 0}^{\infty} a_n s^n \right)^{k_1^2 + ... + k_m^2} \prod_{(i,j) \in W} \left(\frac{1}{s -  i(\alpha_i - \alpha_j)} + \sum_{n = 0}^{\infty} a_n (s - i(\alpha_i - \alpha_j))^n \right)^{k_ik_j} \\
&& \cdot \prod_{(i,j) \in \widetilde{W}} \zeta^{k_ik_j}(s + 1 - i(\alpha_i - \alpha_j))\frac{G(s)x^s}{s}.
\end{eqnarray*}
We claim that as $x \rightarrow \infty,$ the main contribution of $\int_{C'} H(s)\frac{x^s}{s} \> ds$ is 
\begin{equation} \label{maincont}
\int_{C'} \frac{1}{s^{k_1^2 + ... + k_m^2}} \prod_{\substack{(i,j) \in W \\ i < j}}\frac{1}{(s^2 +  (\alpha_i - \alpha_j)^2)^{k_ik_j}}\prod_{(i,j) \in \widetilde{W}} \zeta^{k_ik_j}(s + 1 - i(\alpha_i - \alpha_j))\frac{G(s)x^s}{s} \> ds 
\end{equation}
By Cauchy's theorem, we obtain that for $n \geq 1$ and $b_{ij} \geq 0,$
\begin{eqnarray*}
&& \int_{C'} s^{n} \prod_{(i,j) \in W}(s -  i(\alpha_i - \alpha_j))^{b_{ij}}  \prod_{(i,j) \in \widetilde{W}} \zeta^{k_ik_j}(s + 1 - i(\alpha_i - \alpha_j)) \frac{G(s)x^s}{s} \> ds  = 0.
\end{eqnarray*}
Hence to prove the claim above, it is sufficient to prove the following two inequalities. For $k_1^2 + .. k_m^2 > n \geq 0$,
 \begin{eqnarray*} \label{eqn:smallpowerofS}
&& \int_{C'} \frac{1}{s^n} \prod_{i \neq j}  \zeta^{k_ik_j}(s + 1 - i(\alpha_i - \alpha_j))\frac{G(s)x^s}{s} \> ds  \ll (\log x)^{n + \sum_{(i,j) \in W} k_ik_j} \prod_{(i,j) \in \widetilde{W}} \frac{1}{|\alpha_i - \alpha_j|^{k_ik_j}},
\end{eqnarray*}
and for $1 \leq b_{pq} < k_pk_q.$
\begin{eqnarray*} 
&& \int_{C'} \frac{1}{(s - i(\alpha_p - \alpha_q))^{b_{pq}}} \zeta^{k_1^2 + ... + k_m^2}(s+ 1) \prod_{(i, j) \neq (p,q)} \zeta^{k_ik_j}(s + 1 - i(\alpha_i - \alpha_j)) \frac{G(s)x^s}{s} \> ds \\ \nonumber
&\ll& (\log x)^{k_1^2 + ... + k_m^2 + b_{pq} + \sum_{\substack{(i,j) \in W \\ (i,j)  \neq (p,q)}} k_ik_j} \prod_{(i,j) \in \widetilde{W}} \frac{1}{|\alpha_i - \alpha_j|^{k_ik_j}}, \nonumber
\end{eqnarray*}
The proof of both inequalities above follows easily from the fact that for  $s$ on $C'$,
$$|\zeta(s + 1)| \ll \log x,$$
and 
$$|\zeta(s + 1 - i(\alpha_i - \alpha_j))| \ll
\left\{ \begin{array}{ll}
                    \log x & {\rm if} \,\,\, (i,j) \in W, \\
                     \frac{1}{|\alpha_i - \alpha_j|} & {\rm if} \,\,\, (i,j) \in \widetilde{W}.
                 \end{array} \right.
$$
Now we compute the contribution of (\ref{maincont}) which also equals
\begin{equation} \label{intoverC}
\int_{C'} \frac{1}{s^{V+ 1}} \left(1 + \sum_{n = 1}^{\infty} \frac{b_n}{s^{2n}}\right) \widetilde{G}(s) x^s \> ds, 
\end{equation}
where $V := k_1^2 + ... + k_m^2 + \sum_{\substack{(i,j) \in W \\ i < j}} 2k_ik_j,$ 
$$
\prod_{\substack{(i,j) \in W \\ i < j}}  \left( 1 + \sum_{n = 1}^{\infty} (-1)^n{k_ik_j + n \choose k_ik_j} \frac{(\alpha_i - \alpha_j)^{2n}}{s^{2n}}\right) = 1 + \sum_{n = 1}^{\infty} \frac{b_n}{s^{2n}} = \sum_{n = 0}^{\infty} \frac{b_n}{s^{2n}} ,
$$
and 
$$ \widetilde{G}(s) = \prod_{(i,j) \in \widetilde{W}} \zeta^{k_ik_j}(s+ 1 + i(\alpha_i - \alpha_j)) G(s).$$ 
Notice that $\widetilde{G}(s)$ is analytic on and inside $C'$, and its radius of convergence is $\gg 1/\log x.$ Therefor for $s = O(1/\log x),$ we can write the Taylor series of $\widetilde{G}(s)$ as 
$ \widetilde{G}(s) = \sum_{n = 0}^{\infty} g_n s^n.$ 
For $n \geq 0,$
$$
\int_{C'} \frac{1}{s^{V + 1}}\frac{b_n}{s^{2n}} \widetilde{G}(s) x^s \> ds 
= g_0 b_n \frac{(\log x)^{V + 2n}}{(V + 2n)!} + \sum_{l = 1}^{V + 2n } g_l b_n \frac{(\log x)^{V + 2n - l}}{(V + 2n - l)!}.
$$
Hence (\ref{intoverC}) equals
\begin{eqnarray} \label{eqn:allresidue}
&& g_0 (\log x)^V \sum_{n = 0}^{\infty} \frac{b_n (\log x)^{2n}}{(V + 2n)!} + \sum_{k = 1}^{V} g_k (\log x)^{V - k} \sum_{n = 0}^{\infty} \frac{b_n (\log x)^{2n}}{(V + 2n - k)!} \\ \nonumber
&+& \sum_{l = 1}^{\infty} g_{V + 2l} \sum_{n = 0}^{\infty} \frac{b_{n + l} (\log x)^{2n}}{(2n)!} + \sum_{l = 1}^{\infty} g_{V + 2l - 1} \sum_{n = 0}^{\infty} \frac{b_{n + l } (\log x)^{2n + 1}}{(2n+ 1)!}. \nonumber
\end{eqnarray}
We claim that only the first term of the equation above gives the main contribution. To show that, we first need to find an upper bound for $b_n$ and $g_k.$ Let $M := \max_{(i,j) \in W} \{k_ik_j\},$ $2w :=$ the size of W, and $\alpha = \max_{(i,j) \in W} |\alpha_i - \alpha_j|.$ We have
$$b_n = (-1)^n \sum_{\substack{d_{ij} : \sum_{W} d_{ij} = n \\ d_{ij} \geq 0} } \,\,\, \prod_{\substack{(i,j) \in W \\ i < j}}  {k_ik_j + d_{ij} \choose k_ik_j}(\alpha_i - \alpha_j)^{2d_{ij}}. $$
Since there are ${w + n -1 \choose w-1 }$ terms such that $\sum_{\substack{(i,j) \in W \\ i < j}}  d_{ij} = n,$ where $d_{ij} \geq 0,$ we obtain that for large $T $ and $n \geq 1,$ 
\begin{equation} \label{eqn:bn}
|b_n| \leq  {w + n - 1 \choose w - 1}{M + n \choose M}^w \alpha^{2n} \leq c_0 n^{p} \alpha^{2n},
\end{equation}
where $c_0, p$ depends on $w, M.$ 

Next, let $r$ be the radius of convergence of $\widetilde{G}(s)$. Note $1/\log x = o(r)$. Hence $\lim_{n \rightarrow \infty} \frac{g_{n + 1}}{g_n} = \frac{1}{r} = o(\log x),$ and 
\begin{equation} \label{eqn:gn}
g_{n} \leq g_0 c_1 \left( \frac{2}{r} \right)^n \ll \prod_{(i,j) \in \widetilde{W}} \frac{1}{|\alpha_i - \alpha_j|^{k_ik_j}} \left( \frac{2}{r} \right)^n = o\left((\log x)^n \prod_{(i,j) \in \widetilde{W}} \frac{1}{|\alpha_i - \alpha_j|^{k_ik_j}}\right),
\end{equation}  
where $c_1 \in {\mathbb R}$ depends on $\widetilde{G}.$

By (\ref{eqn:bn}), (\ref{eqn:gn}) and the fact that $b_0 = 1$, the second sum of (\ref{eqn:allresidue}) is bounded above by
\begin{eqnarray*}
&\ll& \sum_{k = 1}^{V} |g_0| \big(\tfrac{2}{r}\big)^{k}(\log x)^{V-k} \left( \frac{1}{(V - k)!} \sum_{n = 1}^{\infty} \frac{ n^p (\alpha \log x)^{2n}}{(V + 2n - k)!} \right) \\
&\leq& \sum_{k = 1}^{V} |g_0| \big(\tfrac{2}{r}\big)^{k}(\log x)^{V-k}\left( \frac{1}{(V - k)!} + \sum_{n = 0}^{\infty} \frac{ n^p (\alpha \log x)^{2n}}{(2n + 1)!} \right). \\
&=& o\left((\log x)^V \prod_{(i,j) \in \widetilde{W}} \frac{1}{|\alpha_i - \alpha_j|^{k_ik_j}}\right),
\end{eqnarray*}
The sum over $n$ inside is $O(1)$ because $|\alpha \log x| \leq \beta$ as  $x \rightarrow \infty$.  
 
Now we consider the third sum of (\ref{eqn:allresidue}). From (\ref{eqn:bn}), we obtain that for any $l \geq 1,$
\begin{equation} \label{eqn:innerforthirdsum}
\sum_{n = 0}^{\infty} b_{n + l} \frac{(\log x)^{2n}}{(2n)!} \leq l^p \alpha^{2l}\sum_{n = 0}^{\infty} \frac{(n + l)^p}{l^p} \frac{(\alpha \log x)^{2n}}{(2n)!} \leq l^p \alpha^{2l}\sum_{n = 0}^{\infty}  \frac{(2\alpha \log x)^{2n}}{(2n)!}. 
\end{equation}
For the last inequality, we use the fact that for all $n \geq 0$ and $l \geq 1,$ $\frac{n + l}{l} \leq n + 1 < 4^n.$ The sum over $n$ is also $O(1)$. By (\ref{eqn:gn}), (\ref{eqn:innerforthirdsum}) and the fact that $\alpha/r = o(1)$, the third sum is 
$$\ll \frac{1}{r^V} \left(\sum_{l = 1}^{\infty} \left( \frac{2\alpha}{r}\right)^{2l} l^p \right)\prod_{(i,j) \in \widetilde{W}} \frac{1}{|\alpha_i - \alpha_j|^{k_ik_j}} = o((\log x)^V \prod_{(i,j) \in \widetilde{W}} \frac{1}{|\alpha_i - \alpha_j|^{k_ik_j}}).$$
Similarly we can conclude that the fourth sum is also $o((\log x)^V \prod_{(i,j) \in \widetilde{W}} \frac{1}{|\alpha_i - \alpha_j|^{k_ik_j}}).$

Finally we consider the first term of (\ref{eqn:allresidue}). By (\ref{eqn:bn}), the sum of the term is 
$$ 1 \ll \frac{1}{V!} + \sum_{n = 1}^{\infty} \frac{b_n \log^{2n}x}{(V + 2n)!} \ll \frac{1}{V!} + \sum_{n = 1}^{\infty} \frac{n^p (\alpha \log x)^{2n}}{(V + 2n)!} = O(1), $$
where the implied constant depends on $\beta$ and ${\bf k}$ since the sum over $n$ depends on $|\alpha\log x| \leq \beta$ for large $x$. Therefore by (\ref{maincont}) - (\ref{eqn:allresidue}), we obtain that
\begin{eqnarray*}
\int_{C'} H(s)\frac{x^s}{s} \> ds &\sim_{{\bf k}, \beta}& (\log x)^V \prod_{(i,j) \in \widetilde{W}} \frac{1}{|\alpha_i - \alpha_j|^{k_ik_j}} \\
&=&  (\log x)^{k_1^2 + ... + k_m^2}\prod_{i < j}\big({\rm min}\{\frac{1}{|\alpha_i - \alpha_j|}, \log x\}\big)^{2k_ik_j}
\end{eqnarray*}
This concludes the proof of (\ref{eqn:mainres}) and Lemma \ref{lem:at}.
\end{proof}

Let $S_1$ be defined as in (\ref{eqn:S1}). Also we define 
$$a_j = \sum_{n_1n_2...n_m = j}d_{k_1}(n_1)d_{k_2}(n_2)...d_{k_m}(n_m)n_1^{-i\alpha_1}n_2^{-i\alpha_2}...n_m^{-i\alpha_m},$$
and 
$$L(s, t) := \zeta\big(\h + it + i\alpha_1 + s \big)^{k_1}\zeta\big(\h + it + i\alpha_2 + s \big)^{k_2}...\zeta\big(\h + it + + s i\alpha_m\big)^{k_m} = \sum_{n=1}^{\infty} \frac{a_n}{n^{1/2 + it + s}}.$$
To calculate the lower bound for $S_1$ we need the following approximation for $L(0,t)$.
\begin{lemma} \label{lem:approxL0} Let $w = 6(k_1 + ... + k_m).$ For large $T$, and $T \leq t \leq 2T,$ 
$$ L(0,t) = \sum_{n \leq T^{2w}} \frac{a_n}{n^{1/2 + it}}e^{-n/T^w} + O\left( \frac{1}{T} \right),$$
where the implied constant depends on ${\bf k} = (k_1,...,k_m).$
\end{lemma}
\begin{proof}
For $c > \h,$ by integrating term by term, we obtain
$$\frac{1}{2\pi i} \int_{(c)} L(s,t) \Gamma(s) x^s \> ds = \sum_{n} \frac{a_n}{n^{1/2 + it}}e^{-n/x}.$$
For the integral on the left hand side, we shift to the vertical line (-1/3). In doing so, we pick up the residues of the integrand at the poles of $L(s, t),$ whose the real part is 1/2, and at $s = 0.$ Therefore, 
\begin{equation} \label{eqn:functionaleqningamma}
\sum_{n} \frac{a_n}{n^{1/2 + it}}e^{-n/x} = L(0,t) + \sum_{j=1}^m {\rm res}_{s = \h - it - i\alpha_i} L(s, t)\Gamma(s) x^s + I,
\end{equation}
where 
$$ I = \int_{(-1/3)}L(s,t) \Gamma(s) x^s \> ds. $$
By Theorem 7 on p 146 in \cite{Te}, 
$$\left|\zeta\left(\tfrac{1}{6} + i(t + \alpha_j + y)\right)\right| \ll |t + \alpha_j + y|^{5/6} \ll |T + y|^{5/6}, $$
since $\alpha_j \ll \log \log T.$ It follows that $L(-1/3 + iy, t) \ll |T + y|^{5(k_1 + ...+ k_m)/6}.$ Also for large $y$, 
$$|\Gamma(-1/3 + iy) \ll |y|^{-5/6}e^{-\pi y/2}.$$
Therefore 
\begin{equation} \label{eqn:I}
I \ll_{{\bf k}} T^{(k_1 + ... + k_m)/2}\left(\frac{T^{k_1 + ... + k_m}}{x} \right)^{1/3} \ll \frac{1}{T^{7(k_1+..+k_m)/6}},
\end{equation}
upon choosing $x = T^w$. 
Furthermore,
\begin{equation} \label{eqn:truncatesum}
\sum_{n} \frac{a_n}{n^{1/2 + it}}e^{-n/x} = \sum_{n \leq T^{2w}} \frac{a_n}{n^{1/2 + it}}e^{-n/T^w} + O\big(\frac{1}{T}\big). 
\end{equation}
Finally we need to show that 
\begin{equation}\label{eqn:resofhalf}
\sum_{j=1}^m {\rm res}_{s = \h - it - i\alpha_i} L(s, t)\Gamma(s) x^s = O\left( \frac{1}{T}\right). 
\end{equation}
Let $J_1, J_2,..., J_n$ be a partition of $\{\alpha_1, \alpha_2, ..., \alpha_m\}$ such that for any $q,$ and any $\alpha_i, \alpha_j \in J_q$, $|\alpha_i - \alpha_j| = O(1/\log T).$ Also we have for each $k$, 
$$ \sum_{\alpha_i \in J_q} {\rm res}_{s = \h - it - i\alpha_i} L(s, t)\Gamma(s) x^s = \oint L(s, t)\Gamma(s) x^s \> ds,$$
where we integrate over a circle centered at $\h - it - i\alpha_i$ with radius $c/\log T,$ where $c/\log T > |\alpha_i - \alpha_j|$ for any $\alpha_j \in J_q$. Note that we can choose $\alpha_i$ to be any element in $J_q.$ By Stirling's formula for the Gamma function and since $\zeta(1/2 + it + i\alpha_j + s) \ll \log T$ for any $\alpha_j$ and $s$ on the circle, we obtain that
$$ \oint L(s, t)\Gamma(s) x^s \> ds = O\big( \frac{1}{T} \big),$$ 
and equation (\ref{eqn:resofhalf}) is proved. The lemma follows from (\ref{eqn:functionaleqningamma}) - (\ref{eqn:resofhalf}). 
\end{proof}
Now we are ready to find the lower bound for $S_1.$
\begin{lemma} \label{lem:S1} Let $K(y)$ be a nonnegative bounded function in ${\bf C}^{\infty}(\R)$ and compactly supported in [1,2], and let $x = \sqrt{T}.$ Unconditionally, for large $T$, we have
$$ S_1 \gg_{{\bf k}, \beta} T (\log T)^{k_1^2 + ... + k_m^2}\prod_{i < j}\big({\rm min}\{\frac{1}{|\alpha_i - \alpha_j|}, \log T\}\big)^{2k_ik_j}.$$
\end{lemma}
\begin{proof} From Lemma \ref{lem:approxL0},
\begin{eqnarray*}
S_1 &=& \int_{-\infty}^{\infty} \left( \sum_{m \leq T^{2w}} \frac{a_m}{m^{1/2 + it}}e^{-n/T^w} + O\left( \frac{1}{T} \right) \right) \overline{A(t)} K\left( \frac{t}{T} \right) \> dt \\
&=& T\sum_{\substack{m \leq T^{2w} \\ n \leq x}} \frac{a_m\overline{a_n}}{(mn)^{1/2}}e^{-m/T^w} \hat{K}\left(T\log\left(\frac{m}{n} \right)\right) + O\left(\frac{1}{T} \int_{-\infty}^{\infty} |A(t)|K\left(\frac{t}{T} \right) \> dt \right).
\end{eqnarray*}
By Cauchy-Schwarz and the boundedness of $K,$
\begin{eqnarray*}
\int_{-\infty}^{\infty} |A(t)|K\left(\frac{t}{T} \right) \> dt &\ll& 
\int_{T}^{2T} |A(t)| \> dt \\
&\ll& T^{\h} \left( \int_{T}^{2T} |A(t)|^2 \> dt \right)^{\h} \\
&\ll_{{\bf k}, \beta}& T\left((\log T)^{(k_1^2 + ... + k_m^2)/2}\prod_{i < j}\big({\rm min}\{\frac{1}{|\alpha_i - \alpha_j|}, \log T\}\big)^{k_ik_j} \right),
\end{eqnarray*}
where the last inequality is obtained from Lemma \ref{lem:at}. Hence the error term is bounded by $T^{\epsilon}$ for any small $ \epsilon >  0.$ Now we consider the sum. 
Since $ e^{-n/T^w} \leq 1,$ the sum equals
$$ \sum_{n \leq x} \frac{|a_n|^2}{n} e^{-n/T^w}\hat{K}(0) + O\left(\sum_{\substack{m \neq n \\ n \leq x \\ m \leq T^{2w}}} \frac{|a_m||a_n|}{(mn)^{1/2}}\hat{K}\left(T\log\left(\frac{m}{n} \right)\right)   \right). $$
Since $n \leq x = \sqrt{T},$ $T\log\left(\frac{m}{n} \right) \gg T^{1/4}.$ Also for positive integer $r$, $\hat{K}(\xi) \ll_{r} \frac{1}{\xi^r},$ and $|a_m| \ll_{\epsilon} m^{\epsilon}.$ Therefore,
$$\sum_{\substack{m \neq n \\ n \leq x \\ m \leq T^{2w}}} \frac{|a_m||a_n|}{(mn)^{1/2}}\hat{K}\left(T\log\left(\frac{m}{n} \right)\right) \ll_r \frac{1}{T^r}.$$ 
For the main term, since $1/e \leq e^{-n/T^w}  $ for $n \leq x$ we have 
$$\sum_{n \leq x} \frac{|a_n|^2}{n} e^{-n/T^w}\hat{K}(0) \gg \sum_{n \leq x} \frac{|a_n|^2}{n}  \sim_{{\bf k}, \beta}  (\log T)^{k_1^2 + ... + k_m^2}\prod_{i < j}\big({\rm min}\{\frac{1}{|\alpha_i - \alpha_j|}, \log T\}\big)^{2k_ik_j},$$
by Lemma \ref{lem:at}. This proves the lemma.
 \end{proof}
\section{Appendix}
\subsection{} Here we will show that $C_k$ in (\ref{eqn:koster_moment}) exists. Let 
$$f(z) = 1 - \frac{(\pi z)^2 }{3!} + \frac{(\pi z)^4 }{5!} - ... = \sum_{n = 0}^{\infty} (-1)^n \frac{(\pi z)^{2n}}{(2n + 1)!}.$$ 
It is clear that $f(z) = f(-z)$ for all $z$, and $f(z)$ is an analytic continuation of $\frac{\sin \pi z}{\pi z}.$ For $j \neq l,$ 
$b_{jl} = f(\mu_j - \mu_l). $
Also notice that $1 = b_{jj} = \lim_{a \rightarrow 0} \frac{\sin a}{a} = f(\mu_j - \mu_j).$

The denominator of the right hand side of (\ref{eqn:koster_moment}) has the factor of the form $(\mu_j - \mu_l)^2$. Specifically, $\Delta^2(2\pi\mu_1,...,2\pi\mu_{2k}) = \prod_{j < l} (2\pi\mu_j - 2\pi\mu_l)^2.$ Hence to prove that the limit in (\ref{eqn:koster_moment}) exists, it is enough to show that the numerator also has a factor of $(\mu_j - \mu_l)^2$ as $\mu_j \rightarrow \mu_l,$ where $j,l$ are both elements in $\{1,..., 2k \}$ or both in $\{2k + 1,..., 4k \}.$  

We start by subtracting the first row of the matrix $(b_{jl})$ from the $j^{th}$ row, where $j = 2,..., 2k,$ and divide row $j$ by the factor of $\mu_j - \mu_1.$ Furthermore, we subtract the $(2k + 1)^{th}$ row from the $i^{th}$ row, where $i = 2k+ 2,..., 4k$ and divide row $i$ by the factor of $\mu_j - \mu_{2k + 1}.$ Define this new matrix as $(b_{1,jl}).$ Hence 
$$\det(b_{jl}) = \det(b_{1,jl}) \prod_{j = 2}^{2k}(\mu_j - \mu_1)(\mu_{2k + j} - \mu_{2k + 1}) .$$ 
Since $\lim_{\mu_a \rightarrow \mu_j} \frac{f(\mu_j - \mu_l) - f(\mu_a - \mu_l)}{\mu_j - \mu_a} = f'(\mu_j - \mu_l),$ we obtain that as $\mu_1 \rightarrow \mu_i$ (or $\mu_{2k + 1} \rightarrow \mu_{2k + i},$) $b_{1, il} \rightarrow f'(\mu_i - \mu_l),$ where $i = 2,..., 2k, 2k + 2,..., 4k.$ 

Next we subtract the second row of the matrix $(b_{1,jl})$ from the $j^{th}$ row, where $j = 3,..., 2k,$ and divide row $j$ by the factor of $\mu_j - \mu_2.$ Also we subtract the $(2k + 2)^{th}$ row from the $i^{th}$ row, where $i = 2k+ 3,..., 4k$ and divide row $i$ by the factor of $\mu_j - \mu_{2k + 2}.$ Define this new matrix as $(b_{2,jl}).$ Hence
$$ \det(b_{1,jl}) = \det(b_{2,jl}) \prod_{j = 3}^{2k}(\mu_j - \mu_2)(\mu_{2k + j} - \mu_{2k + 2}), $$
and as $\mu_2 \rightarrow \mu_i$ (or $\mu_{2k + 2} \rightarrow \mu_{2k + i},$) $b_{2,il} \rightarrow f''(\mu_i - \mu_l),$ where $i = 3,..., 2k, 2k + 3,..., 4k.$ 

Now we continue procedures as above. At step $m^{th},$ where $m = 2,..., 2k - 1$, we subtract the $m^{th}$ row of the matrix $(b_{m-1, jl})$ from the $j^{th}$ row, where $j = m + 1, ..., 2k,$ and divide row $j$ by the factor of $\mu_j - \mu_m.$ Also we subtract the $(2k + m)^{th}$ row from the $i^{th}$ row, where $i = 2k+ m + 1,..., 4k$ and divide row $i$ by the factor of $\mu_j - \mu_{2k + m}.$  Define this new matrix as $(b_{m,jl}).$ We conclude that 
$$ \det(b_{jl}) = \det(b_{2k - 1, jl}) \prod_{i < j} (\mu_j - \mu_i).$$
Also as $\mu_i \rightarrow \mu_j,$ where $1 \leq i < j \leq 2k,$ $b_{2k - 1, jl} \rightarrow f^{(j-1)}(\mu_{j} - \mu_l),$ and as $\mu_{2k + i} \rightarrow \mu_{2k + j},$ where $1 \leq i < j \leq 2k,$ $b_{2k - 1, (2k + j,l)} \rightarrow f^{(j-1)}(\mu_{2k + j} - \mu_l).$ 

Next we repeat the steps above on columns instead of rows. This eventually gives
\begin{equation} \label{continuation}
 \det(b_{jl}) = \det(c_{jl}) \prod_{i < j} (\mu_j - \mu_l)^2.
\end{equation}
As $\mu_i \rightarrow \mu_j,$ $\mu_{2k + i} \rightarrow \mu_{2k + j},$ $\mu_{j} \rightarrow 0,$ and $\mu_{2k + j} \rightarrow c/2\pi,$ where $1 \leq i < j \leq 2k,$ we have 
$$c_{jl} \rightarrow \left\{ \begin{array}{ll}
                   f^{j + l - 2}(0)  & {\rm if} \,\,\, 1 \leq j, l \leq 2k, \\
                   f^{j + l - 4k - 2 }(0) & {\rm if} \,\,\, 2k + 1 \leq j, l \leq 4k,    \\                 
f^{j + l - 2k - 2}\big(\tfrac{c}{2\pi}\big)  & {\rm otherwise.}
                 \end{array} \right.$$
From above and (\ref{continuation}), $C_k$ in (\ref{eqn:koster_moment}) exists.

\subsection{} In this section, we will prove (\ref{eqn:negres}). Recall that 
$$W := \{(i,j) \,\, |\,\, \lim_{T \rightarrow \infty} |\alpha_i - \alpha_j|\log T \,\,< \,\, \infty \,\,\, {\rm and} \,\,\, i \neq j\},$$ 
and 
$$\widetilde{W} := \{(i,j) \,\, |\,\, \lim_{T \rightarrow \infty} |\alpha_i - \alpha_j|\log T \,\,= \,\, \infty \,\,\, {\rm and} \,\,\, i \neq j \}.$$
Without loss of generality, we can assume that $(i,j)$ in (\ref{eqn:negres}) is $(1, 2).$ Let $V_0$ be subsets of $\widetilde{W}$ such that $(p,q) \in V_0$ if and only if $|(\alpha_p - \alpha_q) - (\alpha_1 - \alpha_2)| = O(1/\log T).$ It is sufficient to prove that the contribution of 
$$\sum_{(p,q) \in V_0} {\rm res}_{s = i(\alpha_p - \alpha_q)}H(s)\frac{x^s}{s} $$ 
is negligible. By Cauchy's theorem, we have the sum of the residues in (\ref{eqn:negres}) is equal to
$$ \oint \zeta^{k_1^2 + ... + k_m^2}(s+1) \prod_{i \neq j} \zeta^{k_ik_j}(s+1-i(\alpha_i - \alpha_j)) G(s)\frac{x^s}{s} \> ds, $$
where the integral is over a circle $\widetilde{C}$ centered at $i(\alpha_1 - \alpha_2)$ and with radius $c/\log x,$ where $c/\log x > |(\alpha_p - \alpha_q) - (\alpha_1 - \alpha_2)| + 1/\log x$ for $(p,q) \in V_0.$
From Corollary 1.7 in \cite{MV}, if $|s + i\alpha| \leq B,$ where $s + i\alpha \neq 0,$ and $B$ is a positive real number, then $\zeta(1 + s + i\alpha) =  \frac{1}{s + i\alpha} + O(1),$ where the implied constant depends on $B.$ 
Since $1/\log x = o(|\alpha_1 - \alpha_2|),$  for $s$ on the circle $\widetilde{C},$
\begin{equation} \label{eqn:zetasplus1}
\zeta(s+1) \ll \frac{1}{|\alpha_1 - \alpha_2|}.
\end{equation}
If $(i,j) \in W,$ then $|(\alpha_i - \alpha_j) - (\alpha_1 - \alpha_2)| \sim |\alpha_1 - \alpha_2|,$ and for $s$ on the circle $\widetilde{C}$,
\begin{equation} \label{eqn:coorinW}
\zeta(s + 1 - i(\alpha_i - \alpha_j))\ll \frac{1}{|\alpha_1 - \alpha_2|}.
\end{equation} 
If $(i,j)$ is in $V$ then for $s$ on the circle $\widetilde{C}$,
\begin{equation} \label{eqn:coorinwWinV}
\zeta(s + 1 - i(\alpha_i - \alpha_j)) \ll \log x.
\end{equation}
If $(i,j) \in \widetilde{W}$ but is not in $V_0$, we have three cases.
\begin{itemize}
\item Let $V_1 \subset \widetilde{W} \backslash V_0$ such that $(i,j) \in V_1 $ if $\lim_{T \rightarrow \infty} \frac{|\alpha_1 - \alpha_2|}{|\alpha_i - \alpha_j|} < \infty,$ and  $\lim_{T \rightarrow \infty} \frac{\alpha_1 - \alpha_2}{\alpha_i - \alpha_j} \neq 1.$ Then for $s$ on the circle $\widetilde{C}$,
\begin{equation} \label{eqn:contributionV1}
 \zeta(s + 1 - i(\alpha_i - \alpha_j)) \ll \frac{1}{|\alpha_i - \alpha_j|}.
\end{equation}
\item Let $V_2 \subset \widetilde{W} \backslash V_0$ such that $(i,j) \in V_2 $ if $\lim_{T \rightarrow \infty} \frac{|\alpha_1 - \alpha_2|}{|\alpha_i - \alpha_j|} = \infty$. Then for $s$ on the circle $\widetilde{C}$,
\begin{equation} \label{eqn:contributionV2}
\zeta(s + 1 - i(\alpha_i - \alpha_j)) \ll \frac{1}{|\alpha_1 - \alpha_2|}.
\end{equation}
\item Let $V_3 \subset \widetilde{W} \backslash V_0$ such that $(i,j) \in V_3 $ if $\lim_{T \rightarrow \infty} \frac{(\alpha_1 - \alpha_2)}{(\alpha_i - \alpha_j)} = 1$. Then for $s$ on the circle $\widetilde{C}$,
\begin{equation} \label{eqn:contributionV3}
\zeta(s + 1 - i(\alpha_i - \alpha_j)) \ll \frac{1}{|(\alpha_i - \alpha_j) - (\alpha_1 - \alpha_2)|}.
\end{equation}
\begin{rem}Since $(i,j) \in V_3,$ it is clear that $\frac{1}{|(\alpha_i - \alpha_j) - (\alpha_1 - \alpha_2)|} = o(\log x).$
\end{rem}
\end{itemize}
Note that all implied constants above depend on $\overrightarrow{\alpha}.$
From (\ref{eqn:zetasplus1}) - (\ref{eqn:contributionV3}), we obtain that
\begin{eqnarray*}
&& \oint \zeta^{k_1^2 + ... + k_m^2}(s+1) \prod_{i \neq j} \zeta^{k_ik_j}(s+1 -i(\alpha_i - \alpha_j)) G(s)\frac{x^s}{s} \> ds \\
&\ll& \frac{1}{|\alpha_1 - \alpha_2|^{k_1^2 + ... + k_m^2 + \sum_{(i,j) \in W} k_ik_j}} \cdot \prod_{(i,j) \in V_0} (\log x)^{k_ik_j} \cdot \prod_{(i,j) \in V_1 }\frac{1}{|\alpha_i - \alpha_j|^{k_ik_j}} \\
&& \cdot \prod_{(i,j) \in V_2} \frac{1}{|\alpha_1 - \alpha_2|^{k_ik_j}} \cdot\prod_{(i,j) \in V_3} |(\alpha_i - \alpha_j) - (\alpha_1 - \alpha_2)|^{k_ik_j}.
\end{eqnarray*}
\begin{rem} \label{rem:jiandijcannotbeboth}
If $(i,j) \in V_0,$ then $(\alpha_i - \alpha_j) = (\alpha_1 - \alpha_2) + O(1/\log T).$ Therefore $(j,i)$ cannot be in $V_0.$ Similarly, if $(i,j) \in V_3, (j,i) \notin V_3.$
\end{rem}
We know that $\lim_{x \rightarrow \infty} |\alpha_1 - \alpha_2|\log x = \infty.$ Therefore to prove that the right hand side of the above inequality is 
\begin{eqnarray*}
&& o\left((\log x)^{k_1^2 + ... + k_m^2}\prod_{i \neq j}\big({\rm min}\{\frac{1}{|\alpha_i - \alpha_j|}, \log x\}\big)^{k_ik_j}\right) \\
&=& o\left((\log x)^{k_1^2 + ... + k_m^2 + \sum_{(i,j) \in W} k_ik_j}\cdot \prod_{(i,j) \in V_0} \frac{1}{|\alpha_1 - \alpha_2|^{k_ik_j}} \cdot \prod_{(i,j) \in V_1 } \frac{1}{|\alpha_i - \alpha_j|^{k_ik_j}} \cdot  \right. \\
&& \left. \cdot \prod_{(i,j) \in V_2} \frac{1}{|\alpha_i - \alpha_j|^{k_ik_j}} \cdot \prod_{(i,j) \in V_3} \frac{1}{|\alpha_1 - \alpha_2|^{k_ik_j}}\right),
\end{eqnarray*}
it is enough to show that 
\begin{eqnarray} \label{eqn:mainsmall}
&& \left(\frac{1}{|\alpha_1 - \alpha_2|}\right)^{k_1^2 + ... + k_m^2 + \sum_{ (i,j) \in W \cup V_2 } k_ik_j - \sum_{(i,j) \in V_0 \cup V_3} k_ik_j}  \\
&=& o\left((\log x)^{k_1^2 + ... + k_m^2 + \sum_{(i,j) \in W } k_ik_j - \sum_{(i,j) \in V_0} k_ik_j}  \cdot \prod_{(i,j) \in V_2} \frac{1}{|\alpha_i - \alpha_j|^{k_ik_j}}  \right. \nonumber \\
&& \cdot\left. \prod_{(i,j) \in V_3} \frac{1}{|(\alpha_i - \alpha_j) - (\alpha_1 - \alpha_2)|^{-k_ik_j} } \right).\nonumber
\end{eqnarray} 
We start proving the above by showing that 
\begin{equation} \label{eqn:ineqFork}
k_1^2 + ... + k_m^2 + \sum_{(i,j) \in W} k_ik_j - \sum_{(i,j) \in V_0} k_ik_j > 0.
\end{equation}
Here we define a bipartite graph G with $k_1, k_2, ..., k_m, -k_1,..., - k_m$ as its vertices. There is an edge between $k_p$ and $-k_q$ if and only if $(i,j) \in V_0.$ Hence the set of edges of $G$ corresponds to the set $V_0.$ Moreover, let $G$ have $t$ connected components. We claim that $G$ has the following properties:
\begin{enumerate}
\item If $k_i$ and $-k_j$, where $i \neq j$, are in the same component, then $(i,j) \in V_0$, i.e. there is an edge between $k_i$ and $-k_j$.
\item If $k_i$ and $k_j$ (or $-k_i$ and $-k_j$) are in the same component, $|\alpha_i - \alpha_j| = O\big(\frac{1}{\log x}\big),$ i.e. $(i,j) \in W.$
\item $k_i$ and $-k_i$ are not in the same components. 
\item If $k_i, -k_j$ are in the same component, then $-k_i, k_j$ cannot be in the same component. 
\item At least one component in $G$ has only one vertex. 
\end{enumerate}
{\underline{Proof of property (1):}} Since $k_i$ and $-k_j$ are contained in the same components, there are edges $(k_i, -k_{m_1}), (-k_{m_1}, k_{m_2}), (k_{m_2}, -k_{m_3}), ..., (k_{m_{2r}},-k_j)$ connecting $k_i$ to $-k_j.$ This can be intepreted as
$|(\alpha_i - \alpha_{m_1}) - (\alpha_1 - \alpha_2)|, |(\alpha_{m_2} - \alpha_{m_1}) - (\alpha_1 - \alpha_2)|,..., |(\alpha_{m_{2r}} - \alpha_{j}) - (\alpha_1 - \alpha_2)| = O(1/\log x).$ Hence $|\alpha_i - \alpha_{m_2}|, |\alpha_{m_2} - \alpha_{m_4}|, ..., |\alpha_{m_{2(r-1)}} - \alpha_{m_{2r}}| = O(1/\log x).$ This gives that $|\alpha_i - \alpha_{m_{2r}}| = O(1/\log x),$ and we obtain that
$$|(\alpha_i - \alpha_j) - (\alpha_1 - \alpha_2)| \leq |(\alpha_i - \alpha_{m_{2r}})| + |(\alpha_{m_{2r}} - \alpha_{j}) - (\alpha_1 - \alpha_2)| = O(1/\log x). $$ 
This proves the first property. 
\\
\\
{\underline{Proof of property (2):}} We will prove only for a case of $k_i$ and $k_j$ because the same arguments are applied to the proof of the negative sign case. Since $k_i$ and $k_j$ are in the same components, by property (1), there are two edges $(k_i, -k_l)$ and $(-k_l, k_j)$ linking between $k_i$ and $k_j$. This means that $|(\alpha_i - \alpha_l) - (\alpha_1 - \alpha_2)|$ and $|(\alpha_j - \alpha_l) - (\alpha_1 - \alpha_2)| = O(1/\log x).$ Hence $|\alpha_i - \alpha_j| = O(1/\log x).$ 
\\
\\
{\underline{Proof of property (3):}} If $k_i$ and $-k_i$ are contained in the same components, then by the same reasonings as the proof of property (1), we have that
$ |(\alpha_i - \alpha_i) - (\alpha_1 - \alpha_2) | = O(1/\log x),$
which is impossible since $(1, 2) \in \widetilde{W}.$
\\
\\
{\underline{Proof of property (4):}} This follows from Remark \ref{rem:jiandijcannotbeboth}.
\\
\\
{\underline{Proof of property (5):}} Suppose every components had at least two vertices. Then we can find edges $(k_{l_1}, -k_{l_2})$, $(k_{l_2}, -k_{l_3}),(k_{l_3}, - k_{l_4}),..., (k_{l_{r}}, -k_{l_1})$ in $G$ such that $l_i \neq l_j$ for $i \neq j.$ This can be intepreted as  
$$\alpha_{l_{i}} - \alpha_{l_{i + 1}} = \alpha_1 - \alpha_2 + O(1/\log x),$$
for $i = 1,..., r-1$, and
$$ \alpha_{l_{r}} - \alpha_{l_{1}} = \alpha_1 - \alpha_2 + O(1/\log x).$$  
Summing up all equations above, we have
$$ 0 = \alpha_{l_1} - \alpha_{l_1} = r(\alpha_1 - \alpha_2) + O(1/\log x).$$
This contradicts the fact that $(1,2) \in \widetilde{W}.$ 
\\
\\
We are ready to prove (\ref{eqn:ineqFork}). Let $C_1, C_2,...,C_t$ are components of graph $G$. Let $P_j = \frac{1}{2}(\sum_{v_i {\rm \,\, is\,\, a\,\, vertex\,\, in}\,\, C_j } v_i)^2.$ By property (5), at least one of $P_j's$ is  $v_j^2 \geq 1 > 0.$ By property (3), $P_1 + ... + P_t$ contains a term $k_1^2 + ... + k_m^2.$ By definition of $G$ and properties (1) and (4) , the coefficient of $k_ik_j,$ where $(i,j) \in V_0,$ in $P_1 + ... + P_t$ is -1. Finally by property (2), the coefficient of $k_ik_j,$ where $(i,j) \in W $, is 0 or 1. Hence (\ref{eqn:ineqFork}) follows from property (1) - (5) and the fact that 
$$(P_1 + ... + P_t) > \sum_{j: C_j \,\,{\rm \,\,has \,\,at\,\, least\,\, two \,\, vertices}} P_j \geq 0.$$  
\\
Next we define an equivalence relation on $V_3$ as follow: $(i_1,j_1)$ is equivalent to $(i_2, j_2)$ if and only if $0 < \lim_{T \rightarrow \infty} \frac{|(\alpha_{i_1} - \alpha_{j_1}) - (\alpha_1 - \alpha_2)|}{|(\alpha_{i_2} - \alpha_{j_2}) - (\alpha_1 - \alpha_2)|} < \infty.$ Let $V_3$ have $d$ equivalence classes, say $\widetilde{W_1},..., \widetilde{W_d},$ and we let $\{f_1(T), f_2(T),..., f_d(T)\}$ be representatives of each equivalence class. Furthermore, for all $k = 1,.., d-1, \lim_{T \rightarrow \infty} \frac{f_k(T)}{f_{k+1}(T)} = 0.$ Observe that $\frac{1}{\log T} = o(f_{k}(T)),$ and $f_{k}(T) = o(|\alpha_1 - \alpha_2|).$ We will now define a simple graph $G_l$ corresponding to $f_l(T).$ In fact, $G_l$ is defined in a similar way to $G.$ 

$G_l$ has $k_1, k_2, ..., k_m, -k_1,..., - k_m$ as its vertices. There is an edge between $k_p$ and $-k_q$ if and only if $(i,j) \in \widetilde{W}$ and $|(\alpha_i - \alpha_j) - (\alpha_1 - \alpha_2)| = O(f_l(T)).$ Hence the set of edges of $G_l$ corresponds to the set of $V_0 \cup \widetilde{W_1} \cup ... \cup \widetilde{W_l}.$  Moreover, let $G_l$ have $t_l$ connected components. We claim that $G_l$ has the following properties:
\begin{enumerate}
\item If $k_i$ and $-k_j$, where $i \neq j$, are contained in the same components, then $|(\alpha_i - \alpha_j) - (\alpha_1 - \alpha_2)| = O(f_l(T))$, i.e. there is an edge between $k_i$ and $-k_j$.
\item If $k_i$ and $k_j$ (or $-k_i$ and $-k_j$) are contained in the same components, $|\alpha_i - \alpha_j| = O\big(f_l(T)\big).$ 
\item $k_i$ and $-k_i$ are not in the same components. 
\item If $k_i, -k_j$ are in the same components, then $-k_i, k_j$ cannot be in the same components. 
\item At least one component in $G_l$ has only one vertex. 
\end{enumerate}
The proof of properties above of $G_l$ can be shown in the same way as the proof of properties of $G$, and we use the fact that $f_l(T) = o(|\alpha_1 - \alpha_2|).$ From property (1) - (5) of $G_l,$ we can conclude that
\begin{equation} \label{eqn:ineqforGl}
 k_1^2 + ... k_m^2 + \sum_{(i,j) \in W \cup W_l} k_ik_j - \sum_{(i,j) \in V_0 \cup \widetilde{W_1} \cup... \cup \widetilde{W_l} } k_ik_j > 0, 
\end{equation}
where $W_l$ is a subset of $V_2$ such that $|\alpha_i - \alpha_j| = O(f_l(T)).$ 

By (\ref{eqn:ineqFork}) and the fact that $\frac{1}{f_1(T)} = o(\log x),$ we obtain that
\begin{eqnarray*}
 && \left(\frac{1}{f_1(T)}\right)^{k_1^2 + ... k_m^2 + \sum_{(i,j) \in W \cup W_1} k_ik_j - \sum_{(i,j) \in V_0 \cup \widetilde{W_1} } k_ik_j}  \cdot \prod_{(i,j) \in V_2 \backslash W_1} \frac{1}{|\alpha_i - \alpha_j|^{k_ik_j}}   \nonumber \\
&& \cdot \prod_{(i,j) \in V_3 \backslash \widetilde{W_1}} \frac{1}{|(\alpha_i - \alpha_j) - (\alpha_1 - \alpha_2)|^{-k_ik_j} } \\
&=& o\left((\log x)^{k_1^2 + ... + k_m^2 + \sum_{(i,j) \in W} k_ik_j - \sum_{(i,j) \in V_0} k_ik_j}  \cdot \prod_{(i,j) \in V_2} \frac{1}{|\alpha_i - \alpha_j|^{k_ik_j}}  \right. \nonumber \\
&& \cdot\left. \prod_{(i,j) \in V_3} \frac{1}{|(\alpha_i - \alpha_j) - (\alpha_1 - \alpha_2)|^{-k_ik_j} } \right).\nonumber
\end{eqnarray*}
Since $\lim_{T \rightarrow \infty} \frac{f_k(T)}{f_{k+1}(T)} = 0,$ by induction and (\ref{eqn:ineqforGl}), for $l = 1,.., d-1$ we obtain that
\begin{eqnarray*}
 && \left(\frac{1}{f_{l+1}(T)}\right)^{k_1^2 + ... k_m^2 + \sum_{(i,j) \in W \cup W_{l+1}} 2k_ik_j - \sum_{(i,j) \in V_0 \cup \widetilde{W_{1}} \cup ... \cup \widetilde{W_{l+1}} } k_ik_j}  \cdot \prod_{(i,j) \in V_2 \backslash W_{l+1} } \frac{1}{|\alpha_i - \alpha_j|^{k_ik_j}}   \nonumber \\
&& \cdot \prod_{(i,j) \in V_3 \backslash (\widetilde{W_{1}} \cup ... \cup \widetilde{W_{l+1}})} \frac{1}{|(\alpha_i - \alpha_j) - (\alpha_1 - \alpha_2)|^{-k_ik_j} } \\
&=& o\left(\left(\frac{1}{f_l(T)}\right)^{k_1^2 + ... k_m^2 + \sum_{(i,j) \in W \cup W_l} k_ik_j - \sum_{(i,j) \in V_0 \cup \widetilde{W_{1}} \cup ... \cup \widetilde{W_l} } k_ik_j} \cdot  \prod_{(i,j) \in V_2 \backslash W_l} \frac{1}{|\alpha_i - \alpha_j|^{k_ik_j}} \right. \nonumber \\
&& \cdot \left. \prod_{(i,j) \in V_3 \backslash (\widetilde{W_{1}} \cup ... \cup \widetilde{W_l})} \frac{1}{|(\alpha_i - \alpha_j) - (\alpha_1 - \alpha_2)|^{-k_ik_j} } \right).\nonumber
\end{eqnarray*}
Finally since $\frac{1}{|\alpha_1 - \alpha_2|} = o(\frac{1}{f_d(T)}),$ by two equations above, we derive (\ref{eqn:mainsmall}).

\end{document}